\scrollmode
\documentclass[11pt,a4paper]{article}

\usepackage{amsmath}
\usepackage{amsthm}
\usepackage{amssymb}
\usepackage{srcltx}
\usepackage{hyperref}
\usepackage{tikz}

\newtheorem{lemma}{Lemma}[section]
\newtheorem{theorem}[lemma]{Theorem}
\newtheorem{corollary}[lemma]{Corollary}
\newtheorem{proposition}[lemma]{Proposition}

\theoremstyle{definition}

\newtheorem{definition}[lemma]{Definition}
\newtheorem{remark}[lemma]{Remark}

\numberwithin{equation}{section}

\newcommand{\cb}[1]{{\color{blue}#1}}
\newcommand{\leaveout}[1]{}

\makeatletter
\renewcommand{\p@enumii}{}
\makeatother

\newcommand\wtsmash[1]{{\smash{\widetilde{#1}}\rlap{$\phantom{#1}$}}}
\newcommand\whsmash[1]{{\smash{\widehat{#1}}\rlap{$\phantom{#1}$}}}
\newcommand{\overto}[1]{\mathbin{\stackrel{\raisebox{-2pt}{$\scriptstyle{#1}$}}{\to}}}
\newcommand{\overeq}[1]{\mathbin{\stackrel{\raisebox{-2pt}{$\scriptstyle{#1}$}}{\Longleftrightarrow}}}

\newcommand{\ud}{\,{\mathrm d}}

\newcommand{\Div}{{\rm div\,}}
\newcommand{\Grad}{{\rm grad\,}}

\newcommand\R{{\mathbb R}}

\newcommand\K{{\mathbb K}}

\newcommand\rplus{{\R_{+}}}

\newcommand\zero{\set{0}}

\newcommand{\Ascr}{\mathcal A}
\newcommand{\Bscr}{\mathcal B}
\newcommand{\Cscr}{\mathcal C}
\newcommand{\Dscr}{\mathcal D}

\newcommand{\Gscr}{\mathcal G}
\newcommand{\Hscr}{\mathcal H}
\newcommand{\Iscr}{\mathcal I}

\newcommand{\Kscr}{\mathcal K}
\newcommand{\Lscr}{\mathcal L}

\newcommand{\Rscr}{\mathcal R}

\newcommand{\Tscr}{\mathcal T}
\newcommand{\Uscr}{\mathcal U}
\newcommand{\Xscr}{\mathcal X}
\newcommand{\Yscr}{\mathcal Y}
\newcommand{\Zscr}{\mathcal Z}

\newcommand{\Wscr}{\mathcal W}

\newcommand{\dom}[1]{\mathrm{dom}\left(#1\right)}
\newcommand{\range}[1]{\mathrm{ran}\left(#1\right)}

\newcommand{\Ker}[1]{\ker\left(#1\right)}

\newcommand{\re}[0]{\mathrm{Re}\,}%\left(#1\right)}
\newcommand{\Ipdp}[2]{\left\langle #1 , #2 \right\rangle}

\newcommand{\set}[1]{\left\lbrace #1 \right\rbrace}

\newcommand{\bigmid}{\bigm\vert}
\newcommand{\Bigmid}{\Bigm\vert}
\newcommand{\biggmid}{\biggm\vert}

\newcommand{\bi}{\begin{itemize}}
\newcommand{\ei}{\end{itemize}}
\newcommand{\be}{\begin{enumerate}}
\newcommand{\ee}{\end{enumerate}}

\newcommand{\sbm}[1]{\left[\begin{smallmatrix}#1\end{smallmatrix}\right]}

\newcommand{\bbm}[1]{\begin{bmatrix}#1\end{bmatrix}}

\makeatletter

\def\etv{& \hskip-.3em\vrule\hskip-.3em &} % vertical line in sysmatrix
\def\smalletv{&\vrule&} % vertical line in small sysmatrix
\def\smallcrh{\vrule height0pt depth2\ex@ width0pt
\cr\noalign{\hrule}
\vrule height6.5\ex@ depth0pt width0pt}
\newbox\smallstrutbox
\setbox\smallstrutbox=\hbox{\vrule height6pt depth1.5pt width\z@}
\def\smallstrut{\relax\ifmmode\copy\smallstrutbox\else\unhcopy\smallstrutbox\fi}

\newenvironment{sysmatrix}{
\let\|=\etv
\hskip \arraycolsep
\begin{matrix}}
{\end{matrix}
\hskip \arraycolsep
}       % same as classical latex matrix def.

\newenvironment{smallsysmatrix}{\null\,\vcenter\bgroup
\let\|=\smalletv

\def\\{\smallstrut\math@cr}
\restore@math@cr\default@tag
\baselineskip\z@skip \lineskip\z@skip \lineskiplimit\lineskip
\ialign\bgroup\hfil$\m@th\scriptstyle##$\hfil&&\thickspace\hfil
$\m@th\scriptstyle##$\hfil\crcr
\crcr\noalign{\vskip -.3\ex@}%
}{\crcr\noalign{\vskip -.2\ex@}%
\crcr\egroup\egroup\,%
}

\makeatother

\begin{document}

\title{Linear wave systems on $n$-D spatial domains}

\author{Mikael Kurula\footnote{Corresponding author, \emph{mkurula@abo.fi}. Financial support by Ruth och Nils-Erik Stenb\"acks stiftelse and the University of Twente most gratefully acknowledged.}~ and Hans Zwart}

\thispagestyle{empty}

\maketitle

\begin{abstract}
In this paper we study the linear wave equation on an $n$-dimensional spatial domain. We show that there is a boundary triplet associated to the undamped wave equation. This enables us to characterise all boundary conditions for which the undamped wave equation possesses a unique solution non-increasing in the energy. Furthermore, we add boundary inputs and outputs to the system, thus turning it into an impedance conservative boundary control system. 
\end{abstract}


\section{Introduction}

In this paper we study the following linear system associated to the wave equation:
\begin{equation}\label{eq:physPDE}
  \left\{
    \begin{aligned}
        \rho(\xi)\frac{\partial^2 z}{\partial t^2} (\xi,t) &= 
	  \Div\big(T(\xi)\,\Grad z(\xi,t)\big)-\left(Q_i\frac{\partial z}{\partial t}\right)(\xi,t),
	  \quad \xi\in\Omega,~t\geq0,\\
    0 &= \frac{\partial z}{\partial t}(\xi,t)\quad\text{on}~\Gamma_0\times\rplus,\\
    0 &= \nu\cdot\big(T(\xi)\,\Grad z(\xi,t)\big)+\left(Q_b\frac{\partial z}{\partial t}\right)(\xi,t)\quad\text{on}~\Gamma_1\times\rplus,\\
    u(\xi,t) &= \nu\cdot\big(T(\xi)\,\Grad z(\xi,t)\big)\quad\text{on}~\Gamma_2\times\rplus,\\
    y(\xi,t) &= \frac{\partial z}{\partial t}(\xi,t)\quad\text{on}~\Gamma_2\times\rplus, \\
    z(\xi,0) &= z_0(\xi),\quad \frac{\partial z}{\partial t} (\xi,0) = w_0(\xi)\quad \text{on}~\Omega;
    \end{aligned}\right.
\end{equation}
here $\Omega \subset {\mathbb R}^n$ is a bounded
  spatial domain with Lipschitz-continuous boundary
  $\partial\Omega=\overline{\Gamma_0}\cup \overline{\Gamma_1}\cup
  \overline{\Gamma_2}$, with $\Gamma_k\cap \Gamma_\ell = \emptyset$ for $k\neq\ell$. The vector $\nu$ denotes the outward
  normal at the boundary. Furthermore, $z(\xi,t)$ is the deflection from the
  equilibrium position at point $\xi\in\Omega$ and time $t\geq0$, $u$ (the forces on
  $\Gamma_2$) is the \emph{input}, and $y$ (the velocities at
  $\Gamma_2$) is the \emph{output}.  The physical parameters,
  $\rho(\cdot)$  and $T(\cdot)$ denote the mass density and Young's
  elasticity modulus, respectively. The operators $Q_i$ and $Q_b$
  correspond to damping inside the domain $\Omega$ and at a part of its
  boundary, respectively. Typically $Q_i$ and $Q_b$ are point-wise
  multiplication operators, but they need not be.

  Note that we do not assume that the sets
  $\Gamma_k$ are separated, i.e., that $\overline{\Gamma_k}\cap
  \overline{\Gamma_\ell} =\emptyset$, $k \neq \ell$. However, we assume that the
  $\Gamma_k$'s are disjoint open subsets in the relative topology
  of the boundary, and that the boundaries $\partial\Gamma_k$ of the
  $\Gamma_k$'s have surface measure zero.

The wave system is a standard system in control of partial differential equations which has been widely studied before in the literature; see for instance \cite[Section 7.3]{PazyBook}, \cite[Section 11.3.2]{ReRo93}, or \cite[Section XIV.3]{Yosi80} for the zero-input case $u=0$. Among the more recent papers which are closer to our treatment are \cite{AaLuMa13,MaSt06,MaSt07}. Compared to these, we allow a more general spatial domain, a more general boundary damping operator $Q_b$, and spatially varying physical parameters $\rho$ and $T$.

A main difference between our treatment of the wave equation and those cited above is the first-order representation used in this study. We consider the semigroup generator $\sbm{0&\mathrm{div}\\\mathrm{grad}&0}$ rather than the standard $\sbm{0&I\\\Delta&0}$. This makes it possible to associate a boundary triplet to the wave equation (Section \ref{sec:waveeq}) and it turns out that also obtaining previously known results becomes technically simpler with this choice. Using the results obtained for the homogeneous case, we show in Section \ref{sec:cons} that the inhomogeneous system presented above is an impedance passive boundary control system.

The general boundary triplet techniques that we develop generalise e.g.\ \cite[Thm 7.2.4]{JaZwBook} to $n$-dimensional spatial domains, and they are certainly of independent interest as boundary triplets are still being actively used in the study of PDEs; see e.g.\ \cite{GoGoBook,DHMS09,Arl12} and the references therein.

In our analysis of the wave equation, we recover the
  well-known result that the adjoint of the gradient operator, considered as an unbounded operator from $L^2(\Omega)$ into $L^2(\Omega)^n$, is minus the divergence operator, considered as an unbounded operator from $L^2(\Omega)^n$ into $L^2(\Omega)$. Other work making extensive use of the duality between the divergence and the gradient in the analysis of PDEs is \cite{Tro13,Tro14}; this work suggests that there is potential for extending the approach to certain types of non-linearities at the boundary.

We end the introduction with a summary of the structure of the paper. Section \ref{sec:genBT} presents results for characterizing boundary conditions that induce contraction semigroups, assuming the existence of a boundary triplet. In Section \ref{sec:waveeq}, we associate a boundary triplet to the wave equation and show how the results of Section \ref{sec:genBT} can be applied in this case. Section \ref{sec:cons} concerns the interpretation of the wave system as a conservative boundary control system in different ways: with different choices of input/output spaces, and passivity is considered in both the impedance and scattering sense. The paper also contains two appendices, one with Sobolev-space background and one with two general operator-theoretical results. To our knowledge, Theorem \ref{thm:dirichletrange2b} is new.

\section{General results for boundary triplets}\label{sec:genBT}

We begin by adapting the definition \cite[p.\ 155]{GoGoBook} of a boundary triplet for a symmetric operator to the case of a skew-symmetric operator; see also \cite[\S5]{MaSt07}.
\begin{definition}\label{def:btrip}
Let $A_0$ be a densely defined, skew-symmetric, and closed linear operator on a Hilbert space $X$. By a \emph{boundary triplet} for $A_0^*$ we mean a triple $(\Bscr;B_1,B_2)$ consisting of a Hilbert space $\Bscr$ and two bounded linear operators $B_1,B_2:\dom{A_0^*}\to \Bscr$, such that $\sbm{B_1\\B_2}\dom{A_0^*}=\sbm{\Bscr\\\Bscr}$ and for all $x,\wtsmash x\in\dom{A_0^*}$ there holds
\begin{equation}\label{eq:bvsA0star}
  \Ipdp{A_0^*x}{\wtsmash x}_X+\Ipdp{x}{A_0^*\wtsmash x}_X=\Ipdp{B_1 x}{B_2 \wtsmash x}_\Bscr+\Ipdp{B_2 x}{B_1 \wtsmash x}_\Bscr.
\end{equation}
\end{definition}

Indeed, the analogue of \eqref{eq:bvsA0star} is written as follows in \cite[p.\ 155]{GoGoBook}:
$$
  \Ipdp{\Ascr^*x}{\wtsmash x}-\Ipdp{x}{\Ascr^*\wtsmash x}=\Ipdp{\Gamma_1 x}{\Gamma_2 \wtsmash x} - \Ipdp{\Gamma_2 x}{\Gamma_1 \wtsmash x},
$$
and setting $A_0^*=(i\Ascr)^*$, $B_1=\Gamma_1$, and $B_2=i\Gamma_2$ in \eqref{eq:bvsA0star}, we obtain exactly this. From the definition of boundary triplet it immediately follows that the so-called \emph{minimal operator} $A_0$ can be recovered via $A_0=-A_0^*|_{\Ker{B_1}\cap\Ker{B_2}}$; see \cite[p.\ 155]{GoGoBook}.

Let $X$ be a Hilbert space and let $R$ be a (linear) relation in $X$, i.e., a subspace of $X^2$. Then $R$ is called \emph{dissipative} if $\re \Ipdp{r_1}{r_2}_X\leq0$ for all $\sbm{r_1\\r_2}\in R$, and $R$ is \emph{maximal dissipative} if $R$ has no proper extension to a dissipative relation in $X$. The relation $R$ is called \emph{skew-symmetric} if $\re\Ipdp{r_1}{r_2}=0$ for all $\sbm{r_1\\r_2}\in R$, and it is \emph{(maximal) accretive} if $\sbm{I&0\\0&-I}R$ is (maximal) dissipative. An \emph{operator} $A:X\supset \dom A\to X$ is called \emph{(maximal) dissipative}, \emph{(maximal) accretive}, or \emph{skew-symmetric} if its graph $\Gscr(A):=\sbm{I\\A}\dom A$, seen as a relation in $X$, has the corresponding property.
\begin{theorem}\label{thm:divgradprops}
Let $(\Bscr;B_1,B_2)$ be a boundary triplet for $A_0^*$ and consider the restriction $A$ of $A_0^*$ to a subspace $\Dscr$ containing $\Ker{B_1}\cap\Ker{B_2}$. Define a subspace of $\Bscr^2$ by $\Cscr:=\sbm{B_1\\B_2} \Dscr$. Then the following claims are true:
\begin{enumerate}
\item The domain of $A$ can be written
\begin{equation}\label{eq:DomAtriv}
  \dom A=\Dscr=\set{d\in\dom{A_0^*}\bigmid \bbm{B_1 d\\B_2 d} \in \Cscr}.
\end{equation}
\item The operator closure of $A$ is $A_0^*$ restricted to
$$
  \wtsmash\Dscr=\set{d\in\dom{A_0^*}\bigmid \bbm{B_1 d\\B_2 d} \in \overline\Cscr},
$$ 
where $\overline\Cscr$ is the closure of $\Cscr$ in $\Bscr^2$. Actually, $\wtsmash{\Dscr}$ is the closure of $\Dscr$ in $\dom{A_0^*}$, where $\dom{A_0^*}$ is endowed with the graph norm. Furthermore, $A$ is closed if and only if $\Cscr$ is closed. 
\item The adjoint $A^*$ is the restriction of $-A_0^*$ to $\Dscr'$, where 
$$
  \Dscr'=\set{ d'\in\dom{A_0^*}\bigmid \bbm{B_1d'\\B_2 d'} \in \bbm{0&I\\I&0}\Cscr^{\perp}}.
$$
\item The operator $A$ is (maximal) dissipative if and only if $\Cscr$ is a (maximal) dissipative relation in $\Bscr$. Moreover, $A$ is maximal dissipative if and only if there exists a contraction $V$ on $\Bscr$ such that $\Cscr=\Ker{\bbm{I+V&I-V}}$.
\item The operator $A$ is skew-adjoint if and only if $\Cscr=\sbm{0&1\\1&0}\Cscr^{\perp}$. This holds if and only if $\Cscr=\Ker{\bbm{I+V&I-V}}$ for some unitary operator $V$ on $\Bscr$.
\end{enumerate}

It also holds that $A$ is (maximal) accretive if and only if $\Cscr$ is (maximal) accretive. Consequently, $A$ is skew-symmetric if and only if $\Cscr$ is skew-symmetric.
\end{theorem}

In Theorem \ref{thm:divgradprops}, we use the operator $A$ to define a relation
$\Cscr$, but we can also go the other way around: If we start with an
arbitrary $\Cscr\subset\Bscr^2$ and define $A$ as the restriction of
$A_0^*$ to $\dom A$ by the right hand-side in \eqref{eq:DomAtriv}, then by the surjectivity of $\sbm{B_1\\B_2}$, we have $\Cscr=\sbm{B_1\\B_2}\dom{A}$, and hence all statements in the theorem remain true. Similarly, it follows from part (3) that $\sbm{0&I\\I&0}\Cscr^\perp=\sbm{B_1\\B_2}\dom{A^*}$. It is thus shown how to obtain $\Cscr$ from $\dom A$ and vice versa; part (4) also contains a formula that expresses $\Cscr$ in terms of $V$. Conversely, we can recover $V$ from $\Cscr$ as the mapping
$$
  V:e-f\mapsto e+f,\quad \bbm{f\\e}\in\Cscr,\quad \dom V=\bbm{-I&I}\Cscr.
$$
Indeed, if $\Cscr$ is a maximal dissipative relation in $\Bscr$, then
$V$ defined by this formula is a contraction on $\Bscr$; see also Lemma \ref{L:2.4} below.

\begin{proof}
{\em 1}.\/ Denote the set on the right-hand side of \eqref{eq:DomAtriv} by $\whsmash\Dscr$. Then by the definition of $\Cscr$:
$$
  d\in\Dscr\quad\implies\quad \bbm{B_1d\\B_2d}\in\Cscr\quad\implies\quad d\in\whsmash\Dscr.
$$
Conversely by the definitions of $\whsmash\Dscr$ and $\Cscr$, respectively,
$$
\begin{aligned}
  d\in\whsmash\Dscr\quad&\implies\quad\bbm{B_1d\\B_2d}\in\Cscr\quad\implies\quad \exists d'\in\Dscr:~\bbm{B_1d\\B_2d}=\bbm{B_1d'\\B_2d'} \\
  &\implies \quad \exists d'\in\Dscr:~d-d'\in\Ker{\bbm{B_1d\\B_2d}}\subset\Dscr,
\end{aligned}
$$
and for such a $d'$ we have $d=d-d'+d'\in\Dscr$. Thus $\Dscr=\whsmash\Dscr$.

\smallskip\noindent
{\em 2}.\/ It follows from Lemma \ref{lem:closedchar} that
$\dom{\overline A}=\overline{\dom A}=\overline\Dscr$; hence $A$ is a
closed operator if and only if $\Dscr$ is a closed subspace of
$\dom{A_0^*}$. Moreover, by \eqref{eq:DomAtriv} and statement (2) of
Lemma \ref{lem:invimages}, we have that $\overline\Dscr=\wtsmash\Dscr$
and that $\Dscr$ is closed if and only if $\Cscr$ is closed.

\smallskip\noindent
{\em 3}.\/ From $A\subset A_0^*$ and the definition of the minimal operator, we get $-A_0\subset A$ which in turn implies that $A^*\subset -A_0^*$. Then it follows from \eqref{eq:bvsA0star} that $d'\in\dom{A^*}$ if and only if $\sbm{B_2 d'\\B_1 d'}\perp\sbm{B_1d\\B_2 d}$ for all $d\in\Dscr$, and this proves assertion 3.

\smallskip\noindent
{\em 4}.\/ Both claims follow from \cite[Thm 3.1.6]{GoGoBook} and its proof.

\smallskip\noindent
{\em 5}.\/ Since $-A^*,A\subset A_0^*$, it holds that $A^*=-A$ if and only if $\dom{A^*}=\dom A$. By item 3 and \eqref{eq:DomAtriv}, $\dom{A^*}=\dom A$ if $\Cscr=\sbm{0&I\\I&0}\Cscr^{\perp}$. Conversely, if $\dom{A^*}=\dom A$, then by the above formulas connecting $\dom A$, $\Cscr$, and $\dom{A^*}$:
$$
  \Cscr=\bbm{B_1\\B_2}\dom{A}=\bbm{B_1\\B_2}\dom{A^*}=\bbm{0&I\\I&0}\Cscr^\perp.
$$
The other assertion is contained in \cite[Thm 3.1.6]{GoGoBook}.
\end{proof}

Motivated by item (4) of Theorem \ref{thm:divgradprops}, we now specialise Theorem \ref{thm:divgradprops} to the case where $\Cscr$ is the kernel of some $W_B\in\Lscr(\Bscr^2;\Kscr)$, i.e., $W_B$ is a bounded and everywhere-defined linear operator from $\Bscr^2$ into $\Kscr$.
\begin{theorem}\label{thm:maxdissprel}
Let $(\Bscr;B_1,B_2)$ be a boundary triplet for the operator $A_0^*$ on a Hilbert space $X$, let $\Kscr$ be a Hilbert space, and let $W_B=\bbm{W_1&W_2}\in \Lscr(\Bscr^2;\Kscr)$. The following claims are true for the restriction $A:=A_0^*\big|_{\dom A}$ to $\dom A=\Ker{\bbm{W_1&W_2}\sbm{B_1\\B_2}}$:
\begin{enumerate}
\item The operator $A$ is closed.
\item The operator $A$ is (maximal) dissipative if and only if $\Ker{W_B}$ is a (maximal) dissipative relation in $\Bscr$.
\item The adjoint of $A$ is $A^*=-A_0^*\big|_{\dom{A^*}}$, where
\begin{equation}\label{eq:AadjDom}
  \dom{A^*}=\set{ x\in\dom{A_0^*} \bigmid \bbm{B_1x\\B_2 x}\in \overline{\range{\bbm{ W_2^*\\W_1^*}}}}.
\end{equation}
\item The adjoint $A^*$ is dissipative if and only if
\begin{equation}\label{eq:sigmanevandiss}
  W_1 W_2^*+W_2 W_1^*\geq0
\end{equation}
in $\Kscr$. The adjoint is skew-symmetric, i.e., $\re\Ipdp{A^*x}{x}=0$ for all $x\in\dom{A^*}$, if and only if \eqref{eq:sigmanevandiss} holds with equality.
\item The operator $A$ generates a contraction semigroup on $X$ if and only if $A$ is dissipative and \eqref{eq:sigmanevandiss} holds.
\item The operator $A$ generates a unitary group on $X$ if and only if $A$ is skew-symmetric and \eqref{eq:sigmanevandiss} holds with equality.
\end{enumerate}
\end{theorem}
\begin{proof}
The subspace $\Dscr$ of Theorem \ref{thm:divgradprops} is
$$
  \Dscr=\dom A=\Ker{\bbm{W_1&W_2}\sbm{B_1\\B_2}}\supset\Ker{B_1}\cap\Ker{B_2}.
$$
By \eqref{eq:DomAtriv} and the surjectivity of $\sbm{B_1\\B_2}$, it is easy to see that $\Cscr=\Ker{W_B}$.

\smallskip\noindent
{\em 1}.\/ Since $W_B\in\Lscr(\Bscr^2;\Kscr)$, $\Cscr=\Ker{W_B}$ is closed. Now the closedness of $\dom A$ follows from part (2) of Theorem \ref{thm:divgradprops}.

\smallskip\noindent
{\em 2}.\/
This follows from $\Cscr=\Ker{W_B}$ and part (4) of Theorem \ref{thm:divgradprops}.

\smallskip\noindent
{\em 3}.\/
The domain and action of $A^*$ follow directly from part (3) of Theorem \ref{thm:divgradprops}; note that 
\begin{equation}\label{eq:Cperp}
  \bbm{0&I\\I&0}\Cscr^\perp=\overline{\range{\bbm{W_2^*\\ W_1^*}}}.
\end{equation}
\noindent
{\em 4}.\/
Applying Theorem \ref{thm:divgradprops} to $A^*$, using \eqref{eq:Cperp}, we obtain that $A^*$ is dissipative if and only if $\sbm{0&I\\I&0}\Cscr^\perp$ is accretive; note the minus sign in the formula for $A^*$ in item 3. By the continuity of the inner product this holds if and only if $\range{\sbm{W_2^*\\W_1^*}}$ is accretive, but this is true if and only if \eqref{eq:sigmanevandiss} holds, since
$$
  2\re\Ipdp{W_2^*f}{W_1^*f}_\Bscr=\Ipdp{(W_1 W_2^*+W_2 W_1^*)f}{f}_\Kscr,\quad f\in\Kscr.
$$
A trivial modification of the above gives the proof for the skew-symmetric case.

\smallskip\noindent
{\em 5}.\/ Since $A$ is closed by the first item, this follows from the Lumer-Phillips Theorem.
\smallskip

\noindent
{\em 6}.\/ Since $A$ is closed, both $A$ and $A^*$ are skew-symmetric if and only if and only if $A^*=-A$. The claim follows from Stone's theorem. 
\end{proof}

We next introduce a maximality condition, which implies that $A$ is
dissipative if and only if $A^*$ is dissipative. Theorem \ref{thm:JaZw}
is a general boundary triplet analogue of \cite[Thm 7.2.4]{JaZwBook}. This
theorem can be applied to some PDEs on $n$-dimensional spatial
domains to show existence of solutions; see also \cite[\S4.1]{GZM05}. First, however, we need the following lemma:
\begin{lemma}\label{L:2.4}
Let ${\mathcal B}$ and $\Kscr$ be Hilbert spaces, and let
$\bbm{W_1&W_2}\in \Lscr(\Bscr^2;\Kscr)$. Assume that $W_1+W_2$ is
injective, and that 
\begin{equation}\label{eq:rangecond}
  \range{W_1-W_2}\subset\range{W_1+W_2}.
\end{equation}
Then there exists a unique $V \in {\mathcal L}({\mathcal B})$ such that
\begin{equation}\label{eq:1}
  (W_1+ W_2) V = W_1-W_2,
\end{equation}
or equivalently,
\begin{equation}\label{eq:WBrep}
  \bbm{W_1&W_2}=\frac12(W_1+W_2)\bbm{I+V&I-V}.
\end{equation}
Hence, $\Ker{\bbm{W_1&W_2}}=\Ker{\bbm{I+V&I-V}}$ and, moreover, the operator inequality $W_1W_2^* + W_2 W_1^* \geq 0$ holds in $\Kscr$ if and only if $VV^* \leq I$ in $\Bscr$. 
\end{lemma}

We point out that $W_1W_2^* + W_2 W_1^* \geq 0$ can equivalently be written as
\begin{equation}\label{eq:sigmadisschar}
  \bbm{W_1&W_2}\bbm{0&I\\I&0} \bbm{W_1&W_2}^*\geq0.
\end{equation}
\begin{proof}
We first establish the existence and uniqueness of a $V\in\Lscr(\Bscr)$ such that \eqref{eq:1} holds. Since $W_1+W_2$ is injective, there exists a closed left inverse $(W_1+W_2)^{-l}$ defined on $\range{W_1+W_2}\oplus\big(\range{W_1+W_2}\big)^\perp$. Defining
$$
   V := (W_1+W_2)^{-l}(W_1-W_2),
$$
we obtain from (\ref{eq:rangecond}) that $V$ is defined on all of $\Bscr$. By the boundedness of $W_1-W_2$ and the closedness of $(W_1+W_2)^{-l}$, the composition $V$ is closed, and hence $V\in\Bscr$ by the closed graph theorem. Using assumption (\ref{eq:rangecond}), for all $b\in\Bscr$ there exists a $z\in\Bscr$ such that $(W_1-W_2)b=(W_1+W_2)z$, and we obtain (\ref{eq:1}):
$$
\begin{aligned}
  (W_1+W_2)Vb &= (W_1+W_2)(W_1+W_2)^{-l}(W_1-W_2)b \\
    &= (W_1+W_2)(W_1+W_2)^{-l}(W_1+W_2)z \\
    &= (W_1+W_2)z=(W_1-W_2)b.
\end{aligned}
$$
On the other hand, because of the injectivity of $W_1+W_2$, the operator $V$ is uniquely determined by \eqref{eq:1}.

Now assume that $W_1W_2^* + W_2 W_1^* \geq 0$; we prove that $V$ is a contraction. First note that
$$
  (W_1-W_2)^*\left((W_1+W_2)^{-l}\right)^*\subset V^*,
$$
where the left-hand side is defined densely in $\Bscr$ since $(W_1+W_2)^{-l}$ is densely defined and $(W_1-W_2)^*\in\Lscr(\Kscr;\Bscr)$; hence it suffices to show that $(W_1-W_2)^*\left((W_1+W_2)^{-l}\right)^*$ is contractive. As ${\mathcal B}$ and $\Kscr$ are Hilbert spaces and $W_1,W_2$ bounded, we have that $W_1W_2^* + W_2W_1^* \geq 0$ is equivalent to
\begin{equation}
  \label{eq:2}
  \|(W_1-W_2)^* x\|^2 \leq \|(W_1+W_2)^* x\|^2,\quad x\in\Bscr.
\end{equation}
For arbitrary $y \in \dom{\left((W_1+ W_2)^{-l}\right)^*}$, we set $x:=\left((W_1+W_2)^{-l} \right)^* y$ and obtain from (\ref{eq:2}) that $\|(W_1-W_2)^*\left((W_1+W_2)^{-l}\right)^*y\|^2 \leq \|y\|^2$. We conclude that $W_1W_2^* + W_2 W_1^* \geq 0$ implies that $VV^* \leq I$. Conversely, if $VV^*\leq I$, then using \eqref{eq:WBrep} in \eqref{eq:sigmadisschar}, we have
$$
  \bbm{W_1&W_2}\bbm{0&I\\I&0} \bbm{W_1&W_2}^* = \frac12(W_1+W_2)(I-VV^*)(W_1+W_2)^*\geq0.
$$

Finally, it is straightforward to verify that \eqref{eq:1} is equivalent to \eqref{eq:WBrep}; the equality of the kernels then follows from the injectivity of $W_1+W_2$.
\end{proof}

If $W_1+W_2:\Bscr\to\Kscr$ is invertible then \eqref{eq:rangecond} holds. A good choice of $\Kscr$ can sometimes make this possible.

\begin{theorem}\label{thm:JaZw}
Let $A$ and $\bbm{W_1&W_2}$ be the operators in Theorem \ref{thm:maxdissprel}, and assume that \eqref{eq:rangecond} holds. Then the following conditions are equivalent:
\begin{enumerate}
\item The operator $A$ generates a contraction semigroup on $X$.
\item The operator $A$ is dissipative.
\item The operator $W_1+W_2$ is injective and the following operator inequality holds in $\Kscr$:
\begin{equation}\label{eq:sigmanevanchar}
  W_1 W_2^*+W_2 W_1^*\geq0.
\end{equation}
\end{enumerate}
\end{theorem}
\begin{proof}
The Lumer-Phillips Theorem provides the \emph{implication from 1 to 2}. 

We now prove that \emph{assertion 2 implies assertion 3}. By part (2) of Theorem \ref{thm:maxdissprel} we know that ${\mathcal C} = \ker(W_B)$ is dissipative.  So for every $\sbm{h\\k}\in\Ker{W_B}$ there holds
\begin{equation}\label{eq:AdissBC}
 \re\Ipdp {h}{k}_{\Bscr} \leq 0.
\end{equation}
If $ y \in \ker (W_1+W_2)$, then $W_B\sbm{y\\y}=0$ and by \eqref{eq:AdissBC},   $\re\|y\|_{\Bscr}^2\leq0$. Thus $y=0$ and $W_1+W_2$ is injective. By Lemma \ref{L:2.4}, there exists a $V\in\Lscr(\Bscr)$, such that \eqref{eq:WBrep} holds and $\Ker{W_B}=\bbm{I+V&I-V}$.

Let $u\in \Bscr$ be arbitrary and set $y:=Vu$. Then $\sbm{y-u \\ y+u}$ lies in the dissipative $\Ker{\bbm{I+V&I-V}}$ and hence $\|Vu\|^2-\|u\|^2=\re\Ipdp{y-u}{y+u}\leq0$, which proves that $V$ is a contraction. Lemma \ref{L:2.4} gives that (\ref{eq:sigmanevanchar}) holds.

\smallskip\noindent
\emph{Assertion 3 implies assertion 1}.\/ The assumptions of Lemma \ref{L:2.4} are satisfied and in addition \eqref{eq:sigmanevanchar} holds, and so there exists a contraction $V$ on $\mathcal B$ satisfying \eqref{eq:WBrep}. By part (4) of Theorem \ref{thm:divgradprops}, $A_0^*$ restricted to ${\mathcal D} := \{ d \in \mathrm{dom}(A_0^*) \mid \left[ \begin{smallmatrix} B_1 d \\ B_2 d\end{smallmatrix} \right] \in \ker( \left[ I+V, I - V\right] )\}$ is maximal dissipative and thus the infinitesimal generator of a contraction semigroup. From the injectivity of $W_1+W_2$ and (\ref{eq:WBrep}), we see that ${\mathcal D}$ equals $\{ d \in \mathrm{dom}(A_0^*) \mid \left[ \begin{smallmatrix} B_1 d \\ B_2 d\end{smallmatrix} \right] \in \ker( W_B )\}$. Thus $A$ generates a contraction semigroup.
\end{proof}

The boundary triplet that we shall associate to the wave equation in the next section is of the ``pivoted'' type described in the following result, which is also important in the proof of Theorem \ref{thm:dampedBCS} below.
\begin{theorem}\label{thm:pivotproj}
Let $\Bscr$ be a Hilbert space densely and continuously contained in a Hilbert space $\Bscr_0$, let $\Bscr'$ be the dual of $\Bscr$ with pivot space $\Bscr_0$, and let $\Psi:\Bscr'\to\Bscr$ be a unitary operator. Let $b_2$ be a bounded operator from $\dom{A_0^*}$ to $\Bscr'$ and assume that $(\Bscr;B_1,\Psi b_2)$ is a boundary triplet for the operator $A_0^*$ on the Hilbert space $X$.

Let $V_B=\bbm{V_1&V_2}\in\Lscr(\Bscr_0^2;\Kscr)$, where $\Kscr$ is some Hilbert space, and define 
\begin{equation}\label{eq:AscrDef}
  \Ascr:=\set{a\in \dom{A_0^*} \bigmid b_2a\in\Bscr_0~\wedge~\bbm{V_1&V_2}\bbm{B_1\\b_2}a=0}.
\end{equation}
Then the following two conditions are together sufficient for the closure $A$ of the operator $A_0^*\big|_\Ascr$ (closure in the sense of an operator on $X$) to generate a contraction semigroup on $X$:
\begin{enumerate}
\item $\re\Ipdp uv_{\Bscr_0}\leq0$ for all $u,v\in \Bscr_0$ such that $V_1u+V_2v=0$.
\item The following operator inequality holds in $\Kscr$:
\begin{equation}\label{eq:KVineq}
  V_1V_2^*+V_2V_1^*\geq0.
\end{equation}

\end{enumerate}
The operator $A$ generates a unitary group if $\re\Ipdp uv_{\Bscr_0}=0$ for all $\sbm{u\\v}\in\Ker{V_B}$ and $V_1V_2^*+V_2V_1^*=0$. 

Condition 2 is also necessary for $A$ to generate a contraction semigroup (unitary group) on $X$.
\end{theorem}

Here we have changed to a small $b$ in the boundary mapping $b_2$ in order to avoid confusion. The mapping $B_2$ used previously is analogous to $\Psi b_2$ here. If one wanted to try to reduce Theorem \ref{thm:pivotproj} to Theorem \ref{thm:maxdissprel}, then one might try to set $W_1:=V_1\big|_\Bscr$ and $W_2:=V_2 \Psi^*$. However, this does not go through without complications, because $V_2 \Psi^*$ is in general defined only on $\Psi\Bscr_0$, and not bounded from $\Bscr$ into $\Kscr$.
\begin{proof}
By the definition of $A$, $\Ascr$ is dense in $\dom A$, and so the closed operator $A$ is maximal dissipative if and only if $A\big|_\Ascr$ is dissipative and $A\big|_\Ascr^*=A^*$ is dissipative. 

By \eqref{eq:AscrDef}, we have 
\begin{equation}\label{eq:AscrChar}
\begin{aligned}
  a\in\Ascr\quad &\iff\quad \bbm{B_1\\b_2}a\in\bbm{\Bscr\\\Bscr_0}\cap\Ker{\bbm{V_1&V_2}} \\
    &\iff\quad \bbm{B_1\\\Psi b_2}a\in\bbm{\Bscr\\\Psi\Bscr_0}\cap\Ker{\bbm{V_1&V_2\Psi^*}} \\
    &\iff\quad \bbm{B_1\\\Psi b_2}a\in\Ker{\bbm{V_1\big|_\Bscr&V_2\Psi^*\big|_{\Psi\Bscr_0}}}.
\end{aligned}
\end{equation}
From this we see that the space $\Cscr$ in Theorem \ref{thm:divgradprops} is
\begin{equation}\label{eq:Cpivoted}
  \Cscr=\set{\bbm{q\\p}\in\Bscr^2\bigmid \exists \wtsmash p\in\Bscr_0:\quad p=\Psi\wtsmash p,~V_1q+V_2\Psi^*p=0}.
\end{equation}
For $\sbm{q\\p}\in\Cscr$ there holds
$$
  \Ipdp{q}{p}_\Bscr=\Ipdp{q}{\Psi\wtsmash p}_\Bscr=(q,\wtsmash p)_{\Bscr,\Bscr'}=\Ipdp{q}{\wtsmash p}_{\Bscr_0}\leq0,
$$
where we used condition 1. Theorem \ref{thm:divgradprops} now yields that $A_0^*\big|_{\Ascr}$ is dissipative, and by the continuity of the inner product $A$ is also dissipative. The same argument gives that $A$ is skew-symmetric in case $V_1u+V_2v=0$ implies that $\re\Ipdp{u}{v}=0$.

We next calculate $A^*$ and verify that this adjoint is dissipative if and only if \eqref{eq:KVineq} holds. By items 1--3 in Theorem \ref{thm:divgradprops}, the denseness of $\Ascr$ in $\dom A$, and \eqref{eq:AscrChar}, we obtain
$$
\begin{aligned}
  d\in\dom{A^*}\quad&\iff\quad \bbm{\Psi b_2d\\B_1d}\in
    \Bscr^2 \ominus\left( \bbm{B_1\\\Psi b_2}\Ascr \right) \\
  &\iff\quad \bbm{\Psi b_2d\\B_1d}\in 
    \overline{\range{\bbm{(V_1\big|_\Bscr)^\dagger \\ 
      (V_2\Psi^*\big|_{\Psi\Bscr_0})^\dagger}}}^{\Bscr^2},
\end{aligned}
$$ 
where $\dagger$ denotes the adjoint calculated with respect to the inner product in $\Bscr$ instead of that in $\Bscr_0$. Since $A^*=-A_0^*\big|_{\dom{A^*}}$, we obtain from part (4) of Theorem \ref{thm:divgradprops} that $A^*$ is dissipative if and only if $\range{\bbm{(V_2\Psi^*\big|_{\Psi\Bscr_0})^\dagger \\ (V_1\big|_\Bscr)^\dagger}}$ is an accretive relation in $\Bscr$. We finish the proof by verifying that this is indeed the case, assuming \eqref{eq:KVineq}.

It holds that
$$
  \Ipdp{V_1u}{k}_\Kscr = \Ipdp{u}{V_1^*k}_{\Bscr_0} = (u,V_1^*k)_{\Bscr,\Bscr'} 
    = \Ipdp{u}{\Psi V_1^*k}_\Bscr
$$
for all $u\in\Bscr$ and $k\in\Bscr_0$, and thus $(V_1\big|_\Bscr)^\dagger=\Psi V_1^*$. Moreover, $k\in\dom{(V_2\Psi^*\big|_{\Psi\Bscr_0})^\dagger}$ if and only if there exists some $s\in\Bscr$ such that
\begin{equation}\label{eq:pivotadj}
  \Ipdp{V_2 \Psi^*v}{k}_\Kscr = \Ipdp{v}{s}_\Bscr,\quad v\in \Psi\Bscr_0.
\end{equation}
Now assume that $k\in\dom{(V_2\Psi^*\big|_{\Psi\Bscr_0})^\dagger}$ and choose a $s\in\Bscr$ such that \eqref{eq:pivotadj} holds. Then it holds for all $v\in\Psi\Bscr_0$ that
$$
  \Ipdp{\Psi^* v}{s}_{\Bscr_0} = (\Psi^*v,s)_{\Bscr',\Bscr} = \Ipdp{v}{s}_{\Bscr}
  = \Ipdp{V_2\Psi^* v}{k}_\Kscr = \Ipdp{\Psi^* v}{V_2^*k}_{\Bscr_0},
$$
i.e., that $V_2^*k=s\in\Bscr$. Conversely, $V_2^*k\in\Bscr$ implies that $k\in \dom{(V_2\Psi^*\big|_{\Psi\Bscr_0})^\dagger}$, because then we obtain for all $v\in\Psi\Bscr_0$ that
$$
  \Ipdp{v}{V_2^*k}_\Bscr = (\Psi^*v,V_2^*k)_{\Bscr',\Bscr}=\Ipdp{\Psi^*v}{V_2^*k}_{\Bscr_0}
  = \Ipdp{V_2\Psi^*v}{k}_\Kscr.
$$
We conclude that $(V_2\Psi^*\big|_{\Psi\Bscr_0})^\dagger$ is the restriction of $V_2^*$ to
$$
  \dom{(V_2\Psi^*\big|_{\Psi\Bscr_0})^\dagger}=\set{k\in\Kscr\mid V_2^*k\in\Bscr}.
$$
By definition
$$
  \range{\bbm{(V_2\Psi^*\big|_{\Psi\Bscr_0})^\dagger \\ (V_1\big|_\Bscr)^\dagger}}=\set{\bbm{V_2^*k \\ \Psi V_1^*k}\mid k\in\Kscr~\wedge~V_2^*k\in\Bscr}
$$ 
is an accretive relation in $\Bscr$ if and only if for all $k\in\Kscr$ with $V_2^*k\in\Bscr$ it holds that
$$
  \re\Ipdp{V_2^*k}{\Psi V_1^*k}_\Bscr = \re (V_2^*k,V_1^*k)_{\Bscr,\Bscr'}=
  \re\Ipdp{V_2^*k}{V_1^*k}_{\Bscr_0}
  =\re\Ipdp{V_1V_2^*k}{k}_{\Kscr}\geq0.
$$
This is clearly true if \eqref{eq:KVineq} holds. Conversely, if the relation is accretive, then \eqref{eq:KVineq} holds, since $\set{k\in\Kscr\mid V_2^*k\in\Bscr}=\dom{(V_2\Psi^*\big|_{\Psi\Bscr_0})^\dagger}$ which is dense in $\Kscr$.
\end{proof}

We end the section with the following remark: The only implication in the proof of Theorem \ref{thm:pivotproj}, which is not an equivalence, is where dissipativity of $\bbm{V_1&V_2}$ implies dissipativity of $\sbm{\Bscr\\\Bscr_0}\cap\Ker{\bbm{V_1&V_2}}$.  If the intersection is dense in $\Ker{\bbm{V_1&V_2}}$, then the converse implication is also true by the continuity of the inner product. In this case Theorem \ref{thm:pivotproj} gives necessary and sufficient conditions for $A$ to generate a contraction semigroup.

\section{The wave equation}\label{sec:waveeq}

The notation of this section is described in Appendix \ref{sec:sobolev}, with the additional observation that $\Gamma_\bullet$ in the appendix equals $\Gamma_1\cup\Gamma_2$ in this section. In the rest of the article, we throughout assume that $\Omega\subset\R^n$ is a bounded set with Lipschitz-continuous boundary $\partial\Omega$. It is moreover convenient for us to introduce the concept of a ``splitting with thin common boundary'':
\begin{definition}\label{def:tcb}
By a \emph{splitting of $\partial\Omega$ with thin boundaries}, we mean a finite collection of subsets $\Gamma_k\subset\partial\Omega$, such that:
\begin{enumerate}
\item $\bigcup_k\overline\Gamma_k=\partial\Omega$,
\item the sets $\Gamma_k$ are pairwise disjoint,
\item the sets $\Gamma_k$ are open in the relative topology of $\partial\Omega$, and
\item the boundaries of the sets $\Gamma_k$ all have surface measure zero.
\end{enumerate}
\end{definition}

For instance, if the subset $\Gamma_k$ has Lipschitz-continuous boundary, then the surface measure of $\partial\Gamma_k$ is zero. In the sequel, we always assume the boundary $\partial\Omega$ to be split into subsets with thin boundaries. If we furthermore regard $L^2(\Pi)$, $\Pi\subset\partial\Omega$, as the space of $f\in L^2(\partial\Omega)$ that satisfy $f(x)=0$ for almost every $x\in\partial\Omega\setminus\Pi$, then it holds that
$$
  L^2(\partial\Omega)=\bigoplus_k L^2(\Gamma_k)
$$
and we denote the corresponding orthogonal projections by $\pi_k$. If $\set{\Gamma_0,\Gamma_1}$ is a splitting of $\partial\Omega$ with thin boundaries, then
$$
  L^2(\partial\Omega)=L^2(\Gamma_0)\oplus L^2(\Gamma_1)  \oplus
 L^2\big(\partial\Omega\setminus(\Gamma_0\cup\Gamma_1)\big)
      =L^2(\Gamma_0)\oplus L^2(\Gamma_1),
$$
since $\partial\Omega\setminus(\Gamma_0\cup\Gamma_1)=\partial\Gamma_0\cup\partial\Gamma_1$ has zero surface measure; see also \cite[p.\ 427]{TuWeBook}.

The rest of the paper is devoted to a study of the wave equation
\eqref{eq:physPDE}. We will recall the definitions of scattering and
impedance passive and conservative boundary control systems. We shall
also associate two impedance passive boundary control systems to
\eqref{eq:physPDE} using different input- and output spaces; the
flavour is similar to \cite[\S6.2]{MaSt07}. The main step of the proof
is an application of Theorem \ref{thm:pivotproj} to show that
\eqref{eq:physPDE} is governed by a contraction semigroup on
$L^2(\Omega)^{n+1}$ equipped with a modified but equivalent
norm.

For physical reasons the mass density $\rho(\cdot)\in L^\infty(\Omega)$ takes real positive values and Young's modulus $T(\cdot)\in L^\infty(\Omega)^{n\times n}$ satisfies $T(\xi)^*=T(\xi)$ for almost all $\xi\in\Omega$. We make the additional (physically reasonable) assumption that there exists a $\delta>0$, such that $\rho(\xi)\geq\delta$, and $T(\xi)\geq \delta I$ for almost all $\xi\in\Omega$. We let $Q_i$ and $Q_b$ be bounded and accretive operators on $L^2(\Omega)$ and $L^2(\Gamma_1)$, respectively. If damping inside $\Omega$ is absent, then $Q_i=0$, and if there is no damping at the boundary, then $\Gamma_1=\emptyset$.

The assumptions we made on the parameters imply that the following multiplication operator is bounded, self-adjoint, and uniformly accretive on $\sbm{L^2(\Omega)\\L^2(\Omega)^n}$:
\begin{equation}\label{eq:HscrDef}
  \Hscr x := \xi\mapsto\bbm{1/\rho(\xi)&0\\0&T(\xi)}x(\xi), \quad 
    \xi\in\Omega,~x\in \bbm{L^2(\Omega)\\L^2(\Omega)^n}.
\end{equation}
Hence this operator defines an alternative, but equivalent, inner product on $\sbm{L^2(\Omega)\\L^2(\Omega)^n}$ through $\Ipdp{z_1}{z_2}_{\Hscr}:=\Ipdp{\Hscr z_1}{z_2}$, where $\Ipdp{\cdot}{\cdot}$ denotes the standard inner product on $\sbm{L^2(\Omega)\\L^2(\Omega)^n}$. We denote $\sbm{L^2(\Omega)\\L^2(\Omega)^n}$ equipped with the inner product $\Ipdp{\cdot}{\cdot}_\Hscr$ by $\Xscr_{\Hscr}$.

We invite the reader to carry out the straightforward verification that the first two lines of the PDE \eqref{eq:physPDE} correspond to the following abstract ordinary differential equation
\begin{equation}\label{eq:statewave}
  \dot x(t)=(S-Q)\,\Hscr x(t),\quad t\geq0,
\end{equation}
where the dot denotes derivative with respect to time, the state vector $x(t) = \sbm{M_\rho\,\dot z(t)\\\Grad z(t)}$ consists of the infinitesimal momentum and strain at the point $\xi\in\Omega$, $M_\rho$ is the operator in $L^2(\Omega)$ of multiplication by $\rho$,
$$
  S= \bbm{0&\mathrm{div}\\\mathrm{grad}&0}\bigg|_{\dom S},\quad \dom{S} =  \bbm{H^1_{\Gamma_0}(\Omega)\\H^{\mathrm{div}}(\Omega)},\quad \text{and}\quad Q=\bbm{Q_i&0\\0&0}.
$$
Note how the boundary condition on line two of \eqref{eq:physPDE} becomes part of the domain of $S$. When we initialise \eqref{eq:physPDE} with the initial conditions $z(\xi,0)=z_0(\xi)$ and $\dot z(\xi,0)=w_0(\xi)$, $\xi\in\Omega$, then the corresponding initial state for \eqref{eq:statewave} will be $x(0)=\sbm{M_\rho w_0\\\Grad z_0}$. At this point, any constants in $z_0$ disappear, but they can be recovered using Theorem \ref{thm:solution} below.

The operator $-Q\Hscr=\sbm{Q_iM_{1/\rho}&0\\0&0}$ in \eqref{eq:statewave} is dissipative, bounded and defined on all of $\Xscr_{\Hscr}$. By the passive majoration technique in \cite[Thm 3.2]{AaLuMa13}, we may without loss of generality assume that $Q_i=0$ in the sequel.

\subsection{A boundary triplet and contraction semigroups}

We shall associate a boundary triplet to the wave equation. The main objective is to apply the results in Section \ref{sec:genBT} in order to characterise boundary conditions giving a contraction semigroup.

In Appendix \ref{sec:sobolev}, we give the definitions of the Sobolev
  spaces $H^1(\Omega), H^{\mathrm{div}}(\Omega), H^1_{\Gamma_0}(\Omega)$, and
  $H^{1/2}(\partial\Omega)$. Moreover, we define the Dirichlet
  trace $\gamma_0:H^1(\Omega)\to H^{1/2}(\partial\Omega)$, which maps
  $H^1_{\Gamma_0}(\Omega)$ onto $\Wscr\subset
  L^2(\Gamma_1\cup\Gamma_2)$.  Furthermore, we introduce the restricted
  normal trace $\gamma_\perp:H^{\mathrm{div}}(\Omega)\to\Wscr'$, where
  $\Wscr'$ is the dual of $\Wscr$ with pivot space $L^2(\Gamma_1\cup\Gamma_2)$. Note that $\gamma_\perp$ is \emph{not} a Neumann trace $\gamma_N$; if $\Gamma_0=\emptyset$, then $\Wscr=H^{1/2}(\partial\Omega)$ and the relation between the two operators is $\gamma_N x=\gamma_\perp\,\Grad x$, for $x$ smooth enough, where the equality is in $H^{-1/2}(\partial\Omega)$. Finally define the Hilbert space
\begin{equation}\label{eq:GpKern}
  H^{\mathrm{div}}_{\Gamma_0}(\Omega):=\Ker{\gamma_\perp }
\end{equation}
with the norm inherited from $H^{\mathrm{div}}(\Omega)$.

We next show how \eqref{eq:physPDE} is associated to a contraction semigroup on $\Xscr_{\Hscr}$ by setting $u(\cdot,t)=0$ for all $t\geq0$ and disregarding the output equation on the last line of \eqref{eq:physPDE}. To this end, we combine line four (with $u=0$) and line three of \eqref{eq:physPDE} by writing 
$$
  \bbm{Q_b\pi_1 \\0}\gamma_0\,\frac{\partial z}{\partial t}+\gamma_\perp\, T\,\Grad z=0.
$$
where $\pi_1$ is the orthogonal projection of
  $L^2(\Gamma_1\cup\Gamma_2)$ onto $L^2(\Gamma_1)$.
More precisely, we shall show the more general statement that the operator
\begin{equation}\label{eq:A1}
\begin{aligned}
  A_\Hscr&:=S\Hscr\big|_{\dom{A_\Hscr}},\quad \dom {A_{\Hscr}} &:= \set{x\in\Hscr^{-1}\bbm{H^1_{\Gamma_0}(\Omega)\\H^{\mathrm{div}}(\Omega)} 
    \biggmid \bbm{\wtsmash Q_b\gamma_0&\gamma_\perp }\Hscr x =0},
\end{aligned}
\end{equation}
generates a contraction semigroup on $\Xscr_{\Hscr}$, for an arbitrary accretive $\wtsmash Q_b\in\Lscr(\Wscr;\Wscr')$; note that $\wtsmash Q_b\in\Lscr\big(L^2(\Gamma_1\cup\Gamma_2)\big)$ implies that $\wtsmash Q_b\in\Lscr(\Wscr;\Wscr')$.
\begin{theorem}\label{thm:divgradBT}
Let $\Omega$ be a bounded Lipschitz set. The operator
$$
  A_0 := \bbm{0&-\mathrm{div}\\-\mathrm{grad}&0}\Hscr,\quad 
  \dom{A_0} :=
  \Hscr^{-1}\bbm{H_{0}^1(\Omega)\\H_{\Gamma_0}^{\mathrm{div}}(\Omega)}, 
$$ 
is closed, skew-symmetric, and densely defined on $\Xscr_{\Hscr}$. Its adjoint is 
\begin{equation}\label{eq:AzeroAdj}
  A_0^*=\bbm{0&\mathrm{div}\\\mathrm{grad}&0}\Hscr,\quad \dom{A_0^*}=\Hscr^{-1}\bbm{H_{\Gamma_0}^1(\Omega)\\H^{\mathrm{div}}(\Omega)}.
\end{equation}
Let $M_{1/\rho}$ and $M_T$ be the multiplication operators on the
diagonal of $\Hscr$ in \eqref{eq:HscrDef}, 
%where $\rho$ and $T$ are read off from \eqref{eq:physPDE}, 
and set $B_0:=\bbm{\gamma_0M_{1/\rho}&0}$ and $B_\perp:=\bbm{0&\gamma_\perp M_T}$. Then $(\Wscr; B_0,\Psi_\Wscr B_\perp)$ is a boundary triplet for $A_0^*$, where $\Psi_\Wscr:\Wscr'\to\Wscr$ is any unitary operator. In particular,
\begin{equation}\label{eq:divgradbdrclean}
%\begin{aligned}
  \Ipdp{A_0^*x}{\wtsmash x}_{\Xscr_{\Hscr}} + \Ipdp{x}{A_0^*\wtsmash x}_{\Xscr_{\Hscr}} = 
    \Ipdp{B_0x}{\Psi_\Wscr B_\perp \wtsmash x}_{\Wscr} + \Ipdp{\Psi_\Wscr B_\perp x}{B_0\wtsmash x}_{\Wscr}, \quad x,\wtsmash x\in \dom{A_0^*}.
%\end{aligned}
\end{equation}
\end{theorem}
\begin{proof}
That $\sbm{B_0\\\Psi_\Wscr B_\perp} \Hscr^{-1} \sbm{H_{\Gamma_0}^1(\Omega)\\H^{\mathrm{div}}(\Omega)}=\Wscr^2$ follows from $\gamma_0H_{\Gamma_0}^1(\Omega)=\Wscr$ and $\gamma_\perp  H^{\mathrm{div}}(\Omega)=\Wscr'$; see Theorem \ref{thm:dirichletrange2b}.

The identity \eqref{eq:divgradbdrclean} is obtained by polarizing the following consequence of the integration by parts formula \eqref{eq:partintext2}: For all $\sbm{M_{1/\rho}g\\M_T f}\in \sbm{H^1_{\Gamma_0}(\Omega)\\H^{\mathrm{div}}(\Omega)}$, we obtain
\begin{equation}\label{eq:divgradbdrcalc}
\begin{aligned}
  2\re\Ipdp{\bbm{0&\mathrm{div}\\\mathrm{grad}&0}\bbm{M_{1/\rho}\,g\\M_T f}}{\bbm{M_{1/\rho}\,g\\M_Tf}}_{L^2(\Omega)^{n+1}} &= \\ 
  2\re \Big(\Ipdp{\Div M_Tf}{M_{1/\rho}\,g}_{L^2(\Omega)} + \Ipdp{M_T f}{\Grad M_{1/\rho}\,g}_{L^2(\Omega)^n}\Big) &=\\
  2\re (\gamma_\perp  M_T f,\gamma_0 M_{1/\rho}\, g)_{\Wscr',\Wscr} &= \\
  2\re \Ipdp{\Psi_\Wscr B_\perp\bbm{g\\f}}{B_0\bbm{g\\f}}_{\Wscr}
  ;&
\end{aligned}
\end{equation}
in the last equality we also used that $\Psi_\Wscr$ is unitary.

Now the adjoint of $A_0^*$ in \eqref{eq:AzeroAdj} is $-A_0^*$ restricted to $\Ker{B_0}\cap\Ker{\Psi_\Wscr B_\perp}$. This space equals $\Hscr^{-1}\sbm{H_0^1(\Omega)\\H_{\Gamma_0}^{\mathrm{div}}(\Omega)}$ by Lemma \ref{lem:dirichletrange} and \eqref{eq:GpKern}, i.e., $(A_0^*)^*=A_0$, so that $A_0$ is closed and obviously it is also densely defined. Finally $A_0$ is skew-symmetric by \eqref{eq:divgradbdrcalc}.
\end{proof}

An interesting special case is obtained by taking $\rho$ and $T$ identities, in which case $\Xscr_{\Hscr}$ reduces to $L^2(\Omega)^{n+1}$. On the other hand, taking $\Gamma_0=\emptyset$, we get the following special case:
\begin{corollary}
The operator $-\sbm{0&\Div\\\Grad&0}\Hscr\Big|_{\Hscr^{-1}\sbm{H^1_0(\Omega)\\ H^{\mathrm{div}}_0(\Omega)}}$ is closed, symmetric, and densely defined. A boundary triplet for its adjoint $\sbm{0&\Div\\\Grad&0}\Hscr$ is given by 
$$
  \left(H^{1/2}(\partial\Omega);\bbm{\gamma_0M_{1/\rho}&0},\Psi_{1/2}\bbm{0&\gamma_\perp} M_T\right),
$$
where $\Psi_{1/2}:H^{-1/2}(\partial\Omega)\to H^{1/2}(\partial\Omega)$ is some arbitrary unitary operator.
\end{corollary}

The following by-product of Theorem \ref{thm:divgradBT} gives an exact
statement on the duality of the divergence and gradient
operators. Surprisingly, we were unable to find a citation
  of this well-known result.
\begin{corollary}
Let $\Gamma_0\subset\partial\Omega$ be open. Then
$\mathrm{grad}\big|_{H^1_{\Gamma_0}(\Omega)}^*=-\mathrm{div}\big|_{H^{\mathrm{div}}_{\Gamma_0}(\Omega)}$
and $\mathrm{grad}\big|_{H^1_{0}(\Omega)}^*=-\mathrm{div}\big|_{H^{\mathrm{div}}(\Omega)}$.
\end{corollary}
\begin{proof}
Choosing $\Hscr$ to be the identity in Theorem \ref{thm:divgradBT}, we obtain
that $A_0$ is given by
$$
  A_0 = \bbm{0&-\mathrm{div}\big|_{H^{\mathrm{div}}_{\Gamma_0}(\Omega)}\\-\mathrm{grad}\big|_{H^1_{0}(\Omega)}&0}.
$$
Using this expression and equation \eqref{eq:AzeroAdj}, we find that
$A_0^*$ satisfies
$$
  \bbm{0&-\mathrm{grad}\big|^*_{H^1_{0}(\Omega)}\\-\mathrm{div}\big|^*_{H^{\mathrm{div}}_{\Gamma_0}(\Omega)}&0}
  =A_0^* = \bbm{0&\mathrm{div}\big|_{H^{\mathrm{div}}(\Omega)}\\\mathrm{grad}\big|_{H^1_{\Gamma_0}(\Omega)}&0}.
$$
This shows that
$\mathrm{grad}\big|_{H^1_{\Gamma_0}(\Omega)}^*=-\left(\mathrm{div}\big|^*_{H^{\mathrm{div}}_{\Gamma_0}(\Omega)}\right)^*=-\mathrm{div}\big|_{H^{\mathrm{div}}_{\Gamma_0}(\Omega)}$,
since $\mathrm{div}\big|_{H^{\mathrm{div}}_{\Gamma_0}(\Omega)}$ is
closed. Looking at the upper right corners, we find
 $\mathrm{grad}\big|_{H^1_{0}(\Omega)}^*=-\mathrm{div}\big|_{H^{\mathrm{div}}(\Omega)}$; this can also be obtained from the previous equality by taking $\Gamma_0=\partial\Omega$.
\end{proof}

The following theorem gives an example of how the general results in Section \ref{sec:genBT} can be applied to the wave equation.
\begin{theorem}\label{thm:wavesemigroup}
For every accretive $\wtsmash Q_b\in\Lscr(\Wscr;\Wscr')$, the operator $A_{\Hscr}$ in \eqref{eq:A1} generates a contraction semigroup on $\Xscr_{\Hscr}$.
\end{theorem}
\begin{proof}
We use Theorem \ref{thm:JaZw}, and we begin by identifying $W_B$. Since $\Psi_\Wscr$ is injective, 
$$
\begin{aligned}
  \dom A_{\Hscr} &= \set{\bbm{g\\f}\in\Hscr^{-1}\bbm{H^1_{\Gamma_0}(\Omega)\\H^{\mathrm{div}}(\Omega)} \Bigmid \Psi_\Wscr\, \wtsmash Q_b\,\gamma_0\, M_{1/\rho}\, g+\Psi_\Wscr\,\gamma_\perp \, M_T\, f=0} \\
  &= \Ker{\bbm{\Psi_\Wscr \wtsmash Q_b&I_\Wscr}\bbm{B_0\\\Psi_\Wscr\, B_\perp}};
\end{aligned}
$$
hence $W_1=\Psi_\Wscr\, \wtsmash Q_b\big|_\Wscr$ and $W_2=I_\Wscr$ with $\Kscr=\Wscr$. We next verify that these operators satisfy \eqref{eq:rangecond} and \eqref{eq:sigmanevanchar}, starting with the latter. We have for all $k\in\Wscr$ that
\begin{equation}\label{eq:opacc}
%\begin{aligned}
  \Ipdp{(W_1 W_2^*+W_2 W_1^*)k}{k}_\Wscr = 
    2\re\Ipdp{\Psi_\Wscr \,\wtsmash Q_b\,k}{k}_\Wscr
  = 2\re\,(\wtsmash Q_bk,k)_{\Wscr',\Wscr}\geq0.
%\end{aligned}
\end{equation}

Since the operator $\Psi_\Wscr \,\wtsmash Q_b$ is defined everywhere, the calculation \eqref{eq:opacc} shows that $\Psi_\Wscr \wtsmash Q_b$ is maximal accretive on $\Wscr$. This implies that $\Psi_\Wscr\, \wtsmash Q_b+I=W_1+W_2$ is invertible in $\Wscr$; hence \eqref{eq:rangecond} holds and $A_{\Hscr}$ generates a contraction semigroup on $X_\Hscr$ by Theorem \ref{thm:JaZw}. 
\end{proof}

It is as usual straightforward to verify directly that $A_{\Hscr}$ is dissipative if $\wtsmash Q_b$ is dissipative, and so we might as well have used part (5) of Theorem \ref{thm:maxdissprel} instead of Theorem \ref{thm:JaZw}. The preceding result should be compared to \cite[\S3.9]{TuWeBook}.

\section{The wave equation as a conservative boundary control system}\label{sec:cons}

In this section we show that, depending on the choices of
  the input and output spaces, we can interpret \eqref{eq:physPDE} as an impedance passive boundary control system in two different ways. First we briefly recall some central concepts on boundary control systems; see e.g. \cite[\S2]{MaSt07} for more details.
\begin{definition}\label{def:bdrnode}
A triple $(L,K,G)$ of operators is an \emph{(internally well-posed) boundary control system} on the triple $(\Uscr,\Xscr,\Yscr)$ of Hilbert spaces if it has the following properties:
\begin{enumerate}
\item The linear operators $L$, $K$ and $G$ have the same domain $\Zscr\subset\Xscr$ and take values in $\Xscr$, $\Yscr$ and $\Uscr$, respectively. The space $\Zscr$ is endowed with the graph norm of $\sbm{L\\K\\G}$ and it is called the \emph{solution space}.
\item The operator $\sbm{L\\K\\G}$ is closed from $\Xscr$ into $\sbm{\Xscr\\\Yscr\\\Uscr}$.
\item The operator $G$ is surjective.
\item The operator $A:=L|_{\Ker G}$ generates a strongly continuous semigroup on $\Xscr$.
\end{enumerate}
The boundary control system is \emph{strong} if $L$ is a closed operator on $\Xscr$.
\end{definition}

As was proved in \cite[Lemma 2.6]{MaSt06}, a boundary control system $(L,K,G)$ on $(\Uscr,\Xscr,\Yscr)$ with solution space $\Zscr$ has the following solvability property: For all initial states $z_0\in\Zscr$ and input signals $u\in C^2(\rplus;\Uscr)$ compatible with $z_0$, i.e., $u(0)=Gz_0$, the following system has a unique state trajectory $z\in C^1(\rplus;\Xscr)\cap C(\rplus;\Zscr)$ and the corresponding output signal satisfies $y\in C(\rplus;\Yscr)$:
$$
  \dot z(t)=Lz(t),\quad u(t)=Gz(t),\quad\text{and}\quad y(t)=Kz(t),\quad t\geq0,\quad z(0)=z_0.
$$
Thus for those initial conditions and inputs, the above differential equation possesses a unique classical solution.

A boundary control system $\Xi:=(L,K,G)$ is called time-flow invertible if the triple $\Xi^\leftarrow:=(-L,G,K)$, the so called time-flow inverse, is also a boundary control system. The following definition is adapted from \cite[\S\S2--3]{MaSt07}:
\begin{definition}\label{def:pasv}
A boundary control system is \emph{scattering passive} if
\begin{equation}\label{eq:pasv}
  2\re\Ipdp{z}{Lz}_\Xscr+\|Kz\|^2_\Yscr\leq\|Gz\|^2_\Uscr,\quad z\in\Zscr,
\end{equation}
which holds if and only if all the classical solutions described above satisfy
$$
  \|x(T)\|^2+\int_0^T\|y(t)\|^2\ud t\leq \|x(0)\|^2+\int_0^T\|u(t)\|^2\ud t,\quad T\geq0.
$$
A boundary control system $\Xi$ is called \emph{scattering energy preserving} if we have equality in \eqref{eq:pasv}, and if in addition $\Xi^\leftarrow$ is also a scattering-energy preserving boundary control system, then $\Xi$ is called \emph{scattering conservative}.

A boundary control system $(L,K,G)$ with input space $\Uscr$ and output space $\Yscr$ is \emph{impedance passive (impedance conservative)} if there exists a unitary operator $\Psi:\Uscr\to\Yscr$, such that the so-called \emph{external Cayley transform} $(L,\wtsmash K,\wtsmash G)$ is a scattering passive (scattering conservative) boundary control system, where
\begin{equation}\label{eq:extcayley}
   \wtsmash K := \frac{\Psi G-K}{\sqrt2}\quad\text{and}\quad\wtsmash G := \frac{\Psi G+K}{\sqrt2}.
\end{equation}
\end{definition}

The following two results formalise \eqref{eq:physPDE} as an impedance
passive boundary control system in two different ways. First we assume
that $\Gamma_1=\emptyset$, i.e., that there is no damping on the boundary.
\begin{corollary}\label{cor:impcons}
Let $\Zscr:=\Hscr^{-1}\sbm{H^1_{\Gamma_0}(\Omega)\\H^{\mathrm{div}}(\Omega)}$, $L:=\sbm{0&\mathrm{div}\\\mathrm{grad}&0}\Hscr$, $K:=B_0$, and $G:=B_\perp$. Then $(L,K,G)$ is an internally well-posed strong impedance conservative boundary control system with state space $\Xscr_{\Hscr}$, input space $\Uscr=\Wscr'$, and output space $\Yscr=\Wscr$. 
\end{corollary}
\begin{proof}
The system $(L,K,G)$ is an impedance conservative strong boundary control system by Theorem \ref{thm:divgradBT} and \cite[Thm 5.2]{MaSt07}.
\end{proof}

In the case of Corollary \ref{cor:impcons}, the operator $\Psi$ in \eqref{eq:extcayley} converts forces into velocities; here we use $\Psi=\Psi_\Wscr$ in \eqref{eq:dualityopH}. If we instead choose $L^2(\Gamma_2)$ as input and output space, then we can drop the assumption $\Gamma_1=\emptyset$:

\begin{theorem}\label{thm:dampedBCS}
Assume that $\Gamma_k$, $k=0,1,2$, form a splitting of $\partial\Omega$ with thin boundaries and let $L$, $K$, and $G$ be as in Corollary \ref{cor:impcons}. Let $Q_b\in\Lscr\big(L^2(\Gamma_1)\big)$ and define 
\begin{equation}\label{eq:ZqDef}
  \Zscr_Q:=\set{z\in\Hscr^{-1}\bbm{H^1_{\Gamma_0}(\Omega)\\H^{\mathrm{div}}(\Omega)}\bigmid Gz\in L^2(\Gamma_1\cup\Gamma_2)~\wedge~ Q_b\pi_1  K z+ \pi_1 G z=0}
\end{equation}
with the graph norm of $\Xi_Q:=\sbm{L_Q\\K_Q\\G_Q}:=\sbm{L\\\pi_2K\\\pi_2G}\bigg|_{\Zscr_Q}$. 

Then $(L_Q,K_Q,G_Q)$ is an internally well-posed impedance passive boundary control system with state space $\Xscr_\Hscr$ and input/output space $\Uscr_Q=L^2(\Gamma_2)$. This system is impedance conservative if and only if $Q_b$ is skew-adjoint. The system is strong if and only if $\Uscr_Q=\zero$.
\end{theorem}
\begin{proof}
By Definition \ref{def:pasv}, it suffices to verify that the external Cayley transform $\left(L_Q,\frac{1}{\sqrt2}(G_Q-K_Q),\frac{1}{\sqrt2}(G_Q+K_Q)\right)$ of $(L_Q,K_Q,G_Q)$ is a scattering passive (conservative) boundary control system. This can, according to \cite[Prop.\ 2.4]{AaLuMa13}, be achieved by establishing the inequality
\begin{equation}\label{eq:GreenLag}
  2\re\Ipdp{L_Qz}{z}_{\Xscr_\Hscr}+\left\|\frac{G_Q-K_Q}{\sqrt2}z\right\|^2_{L^2(\Gamma_2)}
    -\left\|\frac{G_Q+K_Q}{\sqrt2}z\right\|^2_{L^2(\Gamma_2)}\leq0,\quad z\in\Zscr_Q,
\end{equation}
the surjectivity condition $(G_Q+ K_Q)\Zscr_Q=L^2(\Gamma_2)$, and that $L_Q\big|_{\Ker{G_Q+K_Q}}$ generates a contraction semigroup on $\Xscr$. In order to prove conservativity, we additionally need to show that \eqref{eq:GreenLag} holds with equality, that $(G_Q-K_Q)\Zscr_Q=\Xscr$, and that $-L_Q\big|_{\Ker{G_Q-K_Q}}$ generates a contraction semigroup on $\Xscr_{\Hscr}$. We do this in several steps.

\medskip\noindent
\emph{Step 1 ($(G_Q\pm K_Q)\Zscr_Q=L^2(\Gamma_2)$) and \eqref{eq:GreenLag} holds):} Pick a $u\in L^2(\Gamma_2)$ arbitrarily and extend this $u$ by zero on $\Gamma_1$; denote the result by $\wtsmash u$. Then $\wtsmash u\in L^2(\Gamma_1\cup\Gamma_2)$, and by Theorem \ref{thm:dirichletrange2b} we can find an $f\in H^{\mathrm{div}}(\Omega)$, such that $\gamma_\perp  f=\wtsmash u$. Then $\Hscr^{-1}\sbm{0\\f}\in\Zscr_Q$ and $(G_Q\pm K_Q)\Hscr^{-1}\sbm{0\\f}=\pi_2\gamma_\perp  f=u$.

The left-hand side of \eqref{eq:GreenLag} can for every $z\in\Zscr_Q$ be rewritten as
\begin{equation}\label{eq:conscalc}
\begin{aligned}
  2\re\Ipdp{L_Qz}{z}_{\Xscr}-2\re\Ipdp{G_Qz}{K_Qz}_{L^2(\Gamma_2)} = 
  2\re\Ipdp{Lz}{z}_{\Xscr}-2\re\Ipdp{\pi_2Gz}{\pi_2Kz}_{L^2(\Gamma_2)} &= \\
  2\re\Ipdp{Lz}{z}_{\Xscr}-2\re(B_\perp z,B_0 z)_{\Wscr',\Wscr}+2\re\Ipdp{\pi_1Gz}{\pi_1Kz}_{L^2(\Gamma_1)} &= \\
  2\re\Ipdp{Lz}{z}_{\Xscr}-2\re\Ipdp{\Psi_\Wscr B_\perp z}{B_0 z}_\Wscr+2\re\Ipdp{\pi_1Gz}{\pi_1Kz}_{L^2(\Gamma_1)} &= \\
  2\re\Ipdp{\pi_1Gz}{\pi_1Kz}_{L^2(\Gamma_1)}=-2\re\Ipdp{Q_b\pi_1  K z}{\pi_1Kz}_{L^2(\Gamma_1)} &\leq 0,
\end{aligned}
\end{equation}
where we used \eqref{eq:divgradbdrclean}, \eqref{eq:ZqDef}, and that $Q_b$ is accretive on $L^2(\Gamma_1)$. 

\medskip\noindent
\emph{Step 2 ($L_Q\big|_{\Ker{G_Q+K_Q}}$ generates a contraction semigroup):} We use Theorem \ref{thm:pivotproj} and start by verifying that $L_Q\big|_{\Ker{G_Q+K_Q}}$ is a closed operator on $\Xscr_{\Hscr}$. Let therefore $z_k\in\Ker{G_Q+K_Q}$ tend to $z$ in $\Xscr_{\Hscr}$, so that $\Hscr z_k\to\Hscr z$ in $L^2(\Omega)^{n+1}$. Let moreover $L_Q z_k=\sbm{0&\mathrm{div}\\\mathrm{grad}&0}\Hscr z_k\to v$ in $\Xscr_\Hscr$, hence in $L^2(\Omega)^{n+1}$. By the closedness of $\sbm{0&\mathrm{div}\\\mathrm{grad}&0}$, we have $\Hscr z_k\to \Hscr z$ in $\sbm{H^1(\Omega)\\H^{\mathrm{div}}(\Omega)}$ and $v=\sbm{0&\mathrm{div}\\\mathrm{grad}&0}\Hscr z$. Since $\sbm{K\\G}\Hscr^{-1}$ is bounded from $\sbm{H^1(\Omega)\\H^{\mathrm{div}}(\Omega)}$ into $\sbm{\Wscr\\\Wscr'}$, we have $\sbm{K\\G} z_k\to\sbm{K\\G} z$ in $\sbm{\Wscr\\\Wscr'}$. On the other hand, because $z_k\in \Ker{G_Q+K_Q}$, 
$$
  Gz=\lim_{k\to\infty} (\pi_1+\pi_2)Gz_k=\lim_{k\to\infty} -Q_b\pi_1Kz_k-\pi_2Kz_k=-(Q_b\pi_1+\pi_2)Kz,
$$
where the limits are taken in $\Wscr'$. This shows that $Gz\in L^2(\Gamma_1\cup\Gamma_2)$, $\pi_1Gz+Q_b\pi_1Kz=0$, and $G_Qz+K_Qz=0$. Thus, $z\in\Ker{G_Q+K_Q}$ and $L_Qz=v$, i.e., $L_Q$ is closed.

Furthermore, we have
\begin{equation}\label{eq:BCKer}
  \Ker{G_Q+K_Q}=\Ker{\bbm{Q_b\pi_1&\pi_1\\\pi_2&\pi_2}\bbm{B_0\\B_\perp}}
\end{equation}
and this space equals $\Ascr$ in \eqref{eq:AscrDef} with $A_0^*=L$, $B_1=K$, $b_2=G$, $\Bscr_0=L^2(\Gamma_1\cup\Gamma_2)$, $V_1=\sbm{Q_b\pi_1\\\pi_2}$, $V_2=\sbm{\pi_1\\\pi_2}$, and $\Kscr=\sbm{L^2(\Gamma_1)\\L^2(\Gamma_2)}$. By Theorem \ref{thm:pivotproj}, it is sufficient to show that $\Ker{\bbm{V_1&V_2}}$ is a dissipative relation in $L^2(\Gamma_1\cup\Gamma_2)$ and that $V_1V_2^*+V_2V_1^*\geq0$.

The following verifies that $\sbm{u\\v}\in\Ker{\bbm{V_1&V_2}}\implies \re\Ipdp{u}{v}_{L^2(\Gamma_1\cup\Gamma_2)}\leq0$:
$$	
\begin{aligned}
	\bbm{Q_b\pi_1&\pi_1\\\pi_2&\pi_2}\bbm{u\\v}=0\quad
	&\implies\quad \bbm{\pi_1\\\pi_2}v=-\bbm{Q_b\pi_1\\\pi_2}u \quad\implies\\
	\quad \re\Ipdp{\bbm{\pi_1\\\pi_2}u}{\bbm{\pi_1\\\pi_2}v}&=
	-\re\Ipdp{\bbm{\pi_1u\\\pi_2u}}{\bbm{Q_b\pi_1u\\\pi_2u}}\leq0.
\end{aligned}
$$
Moreover, $V_2^*=\bbm{\Iscr_1&\Iscr_2}:\Kscr\to L^2(\Gamma_1\cup\Gamma_2)$, where $\Iscr_k$ is the appropriate injection, and hence for all $r\in L^2(\Gamma_1)$, $s\in L^2(\Gamma_2)$:
$$
	\re\Ipdp{V_1V_2^*\bbm{r\\s}}{\bbm{r\\s}} = 
	\re\Ipdp{\bbm{Q_b\pi_1(r+s)\\\pi_2(r+s)}}{\bbm{r\\s}}
	=\re\Ipdp{\bbm{Q_br\\s}}{\bbm{r\\s}}\geq0.
$$
Theorem \ref{thm:pivotproj} now completes step 2.

\medskip\noindent
\emph{Step 3 ($\Xi_Q$ is impedance conservative iff $Q_b^*=-Q_b$):} First assume that $\Xi_Q$ is impedance conservative; then \eqref{eq:conscalc} holds with equality. If we can establish that $K\Zscr_Q=\Wscr$, then $\pi_1K\Zscr_Q$ is dense in $L^2(\Gamma_1)$, and it follows from the last equality in \eqref{eq:conscalc} and the boundedness of $Q_b$ that $Q_b^*=-Q_b$. We pick $w\in\Wscr$ arbitrarily and choose $g\in M_\rho H^1_{\Gamma_0}(\Omega)$ such that $\gamma_0 M_{1/\rho}g=w$. Further choosing $f\in M_T^{-1} H^{\mathrm{div}}(\Omega)$, such that $\gamma_\perp M_T f=Q_b \pi_1w$, we obtain $\sbm{g\\f}\in\Zscr_Q$ and $K\sbm{g\\f}=\gamma_0 M_{1/\rho}g=w$.

Now conversely assume that $Q_b^*=-Q_b$. Then we have equality in \eqref{eq:conscalc} and the argument in step 2 (with a few changes of signs) shows that $-L_Q\big|_{\Ker{G_Q-K_Q}}$ generates a contraction semigroup on $\Xscr_{\Hscr}$.

\medskip\noindent
\emph{Step 4 (The remaining claims):} Internal well-posedness, i.e., that $L_Q\big|_{\Ker{G_Q}}$ generates a contraction semigroup, is proved using exactly the same argument as in Step 2, but with $V_1=\sbm{Q_b\pi_1\\0}$, since 
$$
  \Ker{G_Q}=\Ker{\bbm{Q_b\pi_1&\pi_1\\0&\pi_2}}.
$$

It remains to show that $\Xi_Q$ is strong if and only if\ $\Uscr_Q=\zero$. If  $\Uscr_Q=\zero$, then $\sbm{L_Q\\0\\0}=\Xi_Q$, which is a boundary control system by the above, hence closed by Definition \ref{def:bdrnode}; then $L_Q$ is a closed operator. If $\Uscr_Q\neq\zero$, then we can choose a $\mu\in\overline{L^2(\Gamma_2)}\setminus L^2(\Gamma_2)$, where the closure is taken in $\Wscr'$, and pick a sequence $\mu_k\in L^2(\Gamma_2)$ such that $\mu_k\to\mu$ in $\Wscr'$. Next we define a sequence in $\Zscr_Q$ that converges in the graph norm of $L_Q$ by setting $\sbm{g_k\\f_k}:=\Hscr^{-1}\sbm{0\\R\mu_k}$, where $\gamma_\perp^{-r}$ is any bounded right inverse of $\gamma_\perp $. The limit $\Hscr^{-1}\sbm{0\\\gamma_\perp^{-r}\mu}$ of this sequence has the property $G\Hscr^{-1}\sbm{0\\\gamma_\perp^{-r}\mu}=\mu\not\in L^2(\Gamma_1\cup\Gamma_2)$, hence $\sbm{0\\\gamma_\perp^{-r}\mu}\not\in\Zscr_Q$ and so $L_Q$ is not closed.
\end{proof}

The following result connects the classical solutions of \eqref{eq:physPDE} to those of $(L_Q,K_Q,G_Q)$:

\begin{theorem}\label{thm:solution}
Let $u\in C^2\big(\rplus;L^2(\Gamma_2)\big)$ and $z_0,w_0\in L^2(\Omega)$ be such that $\sbm{M_\rho w_0\\\Grad z_0}\in\Zscr_Q$ and $G_Q\sbm{M_\rho w_0\\\Grad z_0}=u(0)$. Then the unique classical solution $\sbm{g\\f}$ of
\begin{equation}\label{eq:impsys}
%\begin{aligned}
  \frac{\mathrm d}{\mathrm dt}\bbm{g(t)\\f(t)}=L_Q \bbm{g(t)\\f(t)},\quad 
  G_Q\bbm{g(t)\\f(t)}=u(t), \quad t\geq0, \quad
  \bbm{g(0)\\f(0)}=\bbm{M_\rho w_0\\\Grad z_0},
%\end{aligned}
\end{equation}
\begin{equation}\label{eq:relaxedsmooth}
\begin{aligned}
  \text{satisfies}\quad M_{1/\rho}\, g &\in C\big(\rplus;H^1_{\Gamma_0}(\Omega)\big)\cap C^1\big(\rplus;L^2(\Omega)\big)\quad\text{and}\\ 
  M_T\,f &\in C\big(\rplus;H^{\mathrm{div}}(\Omega)\big)\cap C^1\big(\rplus;L^2(\Omega)^n\big);
\end{aligned}
\end{equation}
in particular $\Grad (M_{1/\rho}\, g) \in C\big(\rplus;L^2(\Omega)^n\big)$. Defining
$$
  y(t):=K_Q \bbm{g(t)\\f(t)}\quad\text{and}\quad z(\xi,t):=z_0(\xi)+\int_0^t \frac{g(\xi,s)}{\rho(\xi)}\ud s,\quad \xi\in\Omega,~t\geq0,
$$
we obtain that $y\in C\big(\rplus;L^2(\Gamma_2)\big)$,
\begin{equation}\label{eq:smoothnessintsol}
\begin{aligned}
  z &\in C^1\big(\rplus;H^1_{\Gamma_0}(\Omega)\big)\cap C^2\big(\rplus;L^2(\Omega)\big)\quad\text{and} \\
  & \Div \big(M_T\,\Grad z(\cdot)\big)\in C\big(\rplus;L^2(\Omega)\big),
\end{aligned}
\end{equation}
and that $(z,y)$ solves \eqref{eq:physPDE} with $Q_i=0$.
Note that by Proposition \ref{lem:dirichletrange} we 
interpret the boundary mapping, $\nu\cdot\big(T(\cdot)\,\Grad
z(\cdot,t)\big)$, as $\gamma_\perp \big(M_T\,\Grad z(t)\big)$.
\end{theorem}
\begin{proof}
By the standard smoothness property of the state trajectory of a boundary control system mentioned after Definition \ref{def:bdrnode}, $\sbm{g\\f}\in C^1(\rplus;\Xscr_\Hscr)\cap C(\rplus;\Zscr_Q)$, and in particular $\Hscr\sbm{g\\f}\in C\left(\rplus;\sbm{H^1_{\Gamma_0}(\Omega)\\H^{\mathrm{div}}(\Omega)}\right)$, i.e, \eqref{eq:relaxedsmooth} holds. We have 
$$
  y=K_Q \bbm{g\\f}=\pi_2\gamma_0\,M_{1/\rho}\,g
$$
which is in $C\big(\rplus;L^2(\Gamma_2)\big)$ by \eqref{eq:relaxedsmooth} and the continuity of $\pi_2\gamma_0:H^1(\Omega)\to L^2(\Gamma_2)$.

From the definition of $z$, it follows immediately that $\dot z(t)=M_{1/\rho}\,g(t)$ and that (using \eqref{eq:impsys})
$$
  \Grad z(t)=\Grad z_0+\int_0^t\Grad \big(M_{1/\rho}\, g(s)\big)\ud s=\Grad z_0+\int_0^t\dot f(s)\ud s=f(t).
$$
This implies that
$$
  M_\rho\,\ddot z(t)=\dot g(t)=\Div \big(M_T\,f(t)\big)=\Div \big(M_T\,\Grad z(t)\big),
$$
and so $\Div \big(M_T\,\Grad z(\cdot)\big)\in C\big(\rplus;L^2(\Omega)\big)$ since $z\in C^2\big(\rplus;L^2(\Omega)\big)$, which proves \eqref{eq:smoothnessintsol}. Moreover, $(y,z)$ solves \eqref{eq:physPDE} with $Q_i=0$.
\end{proof}

\begin{remark}
By Definition \ref{def:pasv} and Theorem \ref{thm:dampedBCS}, the system
\begin{equation}\label{eq:physPDEscatt}
  \left\{
    \begin{aligned}
        \rho(\xi)\frac{\partial^2 z}{\partial t^2} (\xi,t) &= 
	  \Div\big(T(\xi)\,\Grad z(\xi,t)\big)-\left(Q_i\frac{\partial z}{\partial t}\right)(\xi,t),
	  \quad \xi\in\Omega,~t\geq0,\\
    0 &= \frac{\partial z}{\partial t}(\xi,t)\quad\text{on}~\Gamma_0\times\rplus,\\
    0 &= \nu\cdot\big(T(\xi)\,\Grad z(\xi,t)\big)+\left(Q_b\frac{\partial z}{\partial t}\right)	(\xi,t)\quad\text{on}~\Gamma_1\times\rplus,\\
    \sqrt 2 u(\xi,t) &= \nu\cdot\big(T(\xi)\,\Grad z(\xi,t)\big)+\frac{\partial z}{\partial t}(\xi,t)\quad\text{on}~\Gamma_2\times\rplus,\\
    \sqrt 2 y(\xi,t) &= \nu\cdot\big(T(\xi)\,\Grad z(\xi,t)\big)-\frac{\partial z}{\partial t}(\xi,t)\quad\text{on}~\Gamma_2\times\rplus, \\
    z(\xi,0) &= z_0(\xi),\quad \frac{\partial z}{\partial t} (\xi,0) = w_0(\xi)\quad \text{on}~\Omega,
    \end{aligned}\right.
\end{equation}
is a scattering-passive boundary control system with state $\sbm{\rho(\cdot)\dot z\\\Grad z}$, input $u$, and output $y$, and in particular it is $L^2$-well-posed. The state space is $\Xscr_\Hscr$ and the input/output space is $L^2(\Gamma_2)$. The system \eqref{eq:physPDEscatt} is even scattering conservative if $Q_i=0$ and $Q_b^*=-Q_b$. The statements in Theorem \ref{thm:solution} remain true for the scattering representation if one replaces all occurrences of $G_Q$ and $K_Q$ by $\frac{1}{\sqrt2}(G_Q+K_Q)$ and $\frac{1}{\sqrt2}(G_Q-K_Q)$, respectively. The pair $(z,y)$ then solves \eqref{eq:physPDEscatt} with $Q_i=0$ instead of \eqref{eq:physPDE}.
\end{remark}

The scattering-passive system \eqref{eq:physPDEscatt} fits into the
abstract framework developed for Maxwell's equations in
\cite{StWe12,StWe13}, at least in the case $\Gamma_1=\emptyset$, i.e.,
when there is no damping at the boundary. In a
  forthcoming paper, we shall give more details on this.

\def\cprime{$'$} \def\cprime{$'$}
\providecommand{\bysame}{\leavevmode\hbox to3em{\hrulefill}\thinspace}
\providecommand{\MR}{\relax\ifhmode\unskip\space\fi MR }
\providecommand{\MRhref}[2]{%
  \href{http://www.ams.org/mathscinet-getitem?mr=#1}{#2}
}
\providecommand{\href}[2]{#2}

\appendix

\section{Sobolev-space background}\label{sec:sobolev}

The necessary background for the present article has been compiled in \cite{parahyp_report}. Here we only fix the notation briefly and the reader is referred to \cite{parahyp_report} for more details. We mainly cite \cite{TuWeBook} for convenience; references to the standard sources, such as \cite{SpiBook,NecasBook,AdFuBook,GrisBook,LiMaBook1}, can be found there. To the best of our knowledge, Proposition \ref{prop:dirichletrange2a} and Theorems \ref{thm:partintext2}--\ref{thm:dirichletrange2b} are new for the case $\Wscr'\neq H^{-1/2}(\partial\Omega)$, i.e., for $\Gamma_0$ with positive surface measure.

A \emph{bounded Lipschitz set} is a bounded and open subset $\Omega$ of $\R^n$ which has a Lipschitz-continuous boundary; see \cite[\S 13]{TuWeBook}. By $\Dscr(\Omega)$ we mean the space of test functions on $\Omega$, i.e., functions in $C^\infty(\Omega)$ with compact support contained in $\Omega$, and $\Dscr'(\Omega)$ denotes the set of distributions on $\Omega$. 

\begin{definition}
The \emph{divergence operator} is the operator $\mathrm{div}:\Dscr'(\Omega)^n\to\Dscr'(\Omega)$ given by
$$
  \Div v = \frac{\partial v_1}{\partial x_1} + \ldots + \frac{\partial v_n}{\partial x_n},
$$
and the \emph{gradient operator} is the operator $\mathrm{grad}:\Dscr'(\Omega)\to\Dscr'(\Omega)^n$ defined by
$$
  \Grad w = \left(\frac{\partial w}{\partial x_1}, \ldots, \frac{\partial w}{\partial x_n}\right)^\top.
$$
The \emph{Laplacian} is defined as a linear operator on $\Dscr'(\Omega)$ by $\Delta x:=\Div{(\Grad x)}$.
\end{definition}

The Sobolev space $H^1(\Omega)$ is as usual defined as the space 
$$
  H^1(\Omega):=\set{v\in L^2(\Omega)\mid \Grad v\in L^2(\Omega)^n}
$$ 
equipped with the graph norm of $\mathrm{grad}$. Similarly, we define
$$
  H^{\mathrm{div}}(\Omega):=\set{v\in L^2(\Omega)^n\mid \Div v\in L^2(\Omega)},
$$
equipped with the graph norm of $\mathrm{div}$. These are the maximal domains for which $\mathrm{grad}$ and $\mathrm{div}$ can be considered as operators between $L^2$ spaces.

\begin{definition}\label{def:Hzero}
The closure of $\Dscr(\Omega)$ in $H^1(\Omega)$ is denoted by $H^1_0(\Omega)$ and the closure of $\Dscr(\Omega)^n$ in $H^{\mathrm{div}}(\Omega)$ is denoted by $H^{\mathrm{div}}_0(\Omega)$.
\end{definition}

It is easy to see that $H^1(\Omega)^n\subset H^{\mathrm{div}}(\Omega)\subset L^2(\Omega)^n$ with continuous embeddings. It is well known that $\Dscr(\overline\Omega)^n$, the restrictions to the closure of $\Omega$ of all functions in $C^\infty(\R^n)$, is dense in $L^2(\Omega)$; see e.g.\ \cite[Thm I.1.2.1]{GiRaBook}. Hence, $H^{\mathrm{div}}(\Omega)$ is dense in $L^2(\Omega)^n$, and due to the following lemma the other embedding is also dense:

\begin{lemma}\label{lem:TestDenseHdiv}
Let $\Omega$ be a subset of $\R^n$ with Lipschitz-continuous boundary. Then $\Dscr(\overline\Omega)^n$ is dense in $H^{\mathrm{div}}(\Omega)$. It follows that also $H^1(\Omega)^n$ is dense in $H^{\mathrm{div}}(\Omega)$.
\end{lemma}

For proof, see \cite[Thm I.2.4]{GiRaBook}. If $\Omega$ is a bounded Lipschitz set in $\R^n$, then the outward unit normal vector field is defined for almost all $x\in\partial\Omega$ using local coordinates, and we can define a vector field $\nu$ in a neighbourhood of $\overline\Omega$ that coincides with the outward unit normal vector field for almost every $x\in\partial\Omega$; see \cite[Def.\ 13.6.3]{TuWeBook} and the remarks following. According to \cite[pp.\ 424--425]{TuWeBook}, we have $\nu\in L^\infty(\partial\Omega)^n$.

The space $H^{1/2}(\partial\Omega)$ is the Hilbert space of all functions in $L^2(\partial\Omega)$ with finite $H^{1/2}(\partial\Omega)$ norm, which is given by
\begin{equation}\label{eq:HhalfNorm}
  \|f\|_{H^{1/2}(\partial\Omega)}^2 = \|f\|_{L^2(\partial\Omega)}^2
    + \int_{\partial\Omega} \int_{\partial\Omega} \frac{|f(x)-f(y)|^2}{\|x-y\|_{\R^n}^n} \ud \sigma_x\ud \sigma_y,
\end{equation}
where $\ud\sigma$ is the surface measure on $\partial\Omega$; see \cite[\S4]{parahyp_report} or \cite[pp.\ 422--423]{TuWeBook} for more details. The space $H^{-1/2}(\partial\Omega)$ is the dual of $H^{1/2}(\partial\Omega)$ with pivot space $L^2(\partial\Omega)$; see e.g.\ \cite[\S 2.9]{TuWeBook}. 

The following result is a consequence of \cite[Thm I.1.5]{GiRaBook}:

\begin{lemma}\label{lem:dirichletrange}
For a bounded Lipschitz set $\Omega$, the boundary trace mapping $g\mapsto g|_{\partial\Omega}:\Dscr(\overline\Omega)\to C(\partial\Omega)$ has a unique continuous extension $\gamma_0$ that maps $H^1(\Omega)$ onto $H^{1/2}(\partial\Omega)$. The space $H^1_0(\Omega)$ in Definition \ref{def:Hzero} equals $\set{g\in H^1(\Omega)\mid \gamma_0g=0}$.
\end{lemma}

We call $\gamma_0$ the \emph{Dirichlet trace map}. In the following integration by parts formula, the dot $\cdot$ denotes the inner product in $\R^n$, $p\cdot q = q^\top p$ without complex conjugate:

\begin{lemma}\label{lem:partint}
Let $\Omega$ be a bounded Lipschitz subset of $\R^n$. Then
\begin{equation}\label{eq:partint}
  \Ipdp{\Div f}{g}_{L^2(\Omega)} + \Ipdp{f}{\Grad g}_{L^2(\Omega)^n} =   \int_{\partial\Omega}(\nu\cdot \cb\gamma_0f)\,\cb\gamma_0\overline g\ud\sigma.
\end{equation}
holds for arbitrary $f\in H^1(\Omega)^n$ and $g\in H^1(\Omega)$.
\end{lemma}

For a proof, see \cite[Rem.\ 13.7.2]{TuWeBook}. Note that $f\in H^1(\Omega)^n$ implies that the boundary trace of $f$, $\gamma_0f\in H^{1/2}(\partial\Omega)^n\subset L^2(\partial\Omega)^n$. Moreover, by the above it holds that $\nu\in L^\infty(\partial\Omega)^n$, and hence we obtain that $\nu\cdot \gamma_0 f\in L^2(\partial\Omega)$ for all $f\in H^1(\Omega)^n$.

In the sequel we make the standing assumption that $\Gamma_0,\Gamma_\bullet$ forms a splitting of $\partial\Omega$ with thin boundaries; see Definition \ref{def:tcb}. The result statements remain true if $\Gamma_\bullet$ is further split into subsets with thin boundaries, as we do in \S\S\ref{sec:waveeq}--\ref{sec:cons}. Following \cite[\S13.6]{TuWeBook}, we write
\begin{equation}\label{eq:zerospace}
\begin{aligned}
  H_{\Gamma_0}^1(\Omega) &:= \set{g\in H^1(\Omega)\mid(\gamma_0g)|_{\Gamma_0}=0~\text{in}\ L^2(\Gamma_0)}.
\end{aligned}
\end{equation}
We can also write $H^1_{\Gamma_0}(\Omega)=\Ker{\pi_0\,\gamma_0}$, where $\pi_0$ is the orthogonal projection of $L^2(\partial\Omega)$ onto $L^2(\Gamma_0)$. Since $H^{1/2}(\partial\Omega)$ is continuously embedded in $L^2(\partial\Omega)$ by \eqref{eq:HhalfNorm}, the operator $\pi_0\,\gamma_0:H^1(\Omega)\to L^2(\Gamma_0)$ is bounded; hence $H^1_{\Gamma_0}(\Omega)$ is closed in $H^1(\Omega)$.

Obviously $\gamma_0$ maps $H^1_{\Gamma_0}(\Omega)$ onto $\Wscr:=\gamma_0H^1_{\Gamma_0}(\Omega)$ with inner product inherited from $H^{1/2}(\partial\Omega)$. This space is dense in $L^2(\Gamma_\bullet)$ by \cite[Thm 13.6.10 and Rem.\ 13.6.12]{TuWeBook}, and it is immediate that the inclusion map is continuous. Denote the dual of $\Wscr$ with pivot space $L^2(\Gamma_\bullet)$ by $\Wscr'$. 

By the Riesz representation theorem, there exists a unitary operator $\Psi_\Wscr:\Wscr'\to \Wscr$, such that
\begin{equation}\label{eq:dualityopH}
  (x,z)_{\Wscr',\Wscr} 
    = \Ipdp{\Psi_\Wscr x}{z}_{\Wscr} = \Ipdp{x}{\Psi_\Wscr^*z}_{\Wscr'}
\end{equation}
for all $x\in \Wscr'$ and $z\in \Wscr$; see \cite[p.\ 57]{TuWeBook} and \cite[p.\ 288--289]{MaSt07}. Thus $\Wscr'$ is also a Hilbert space, with inner product
$$
  \Ipdp{u}{v}_{\Wscr'}=\Ipdp{\Psi_\Wscr u}{\Psi_\Wscr v}_{\Wscr},\quad u,v\in \Ipdp{u}{v}_{\Wscr'}.
$$
The operator $\Psi_\Wscr$ can alternatively be characterised as the operator in $\Lscr(\Wscr';\Wscr)$ uniquely determined by
$$
  \Ipdp{\Psi_\Wscr x}{z}_\Wscr=\lim_{n\to\infty}\Ipdp{x_n}{z}_{L^2(\Gamma_\bullet)},\quad x\in\Wscr',~z\in\Wscr,
$$
where $x_n\in L^2(\Gamma_\bullet)$ is an arbitrary sequence converging to $x$ in $\Wscr'$; see \cite[\S2.9]{TuWeBook}.

\begin{proposition}\label{prop:dirichletrange2a}
For a bounded Lipschitz set $\Omega$, the restricted normal trace map $u\mapsto (\nu \cdot u)\big|_{\Gamma_\bullet}:\Dscr(\overline\Omega)^n\to L^2(\Gamma_\bullet)$ has a unique continuous extension $\gamma_\perp $ that maps $H^{\mathrm{div}}(\Omega)$ \emph{into} $\Wscr'$. 
\end{proposition}
\begin{proof}
We follow the argument in \cite[Thm I.2.5]{GiRaBook} with some small modifications. By \eqref{eq:partint}, we have 
$$
\begin{aligned}
\left|\int_{\partial\Omega}(\nu\cdot u)\,\overline \phi\ud\sigma\right|&\leq \left|\Ipdp{\Div u}{\phi}_{L^2(\Omega)}\right| + \left|\Ipdp{u}{\Grad \phi}_{L^2(\Omega)^n}\right|\\
&\leq \|\Div u\|_{L^2(\Omega)}\,\|\phi\|_{L^2(\Omega)} + \|u\|_{L^2(\Omega)^n}\,\|\Grad \phi\|_{L^2(\Omega)^n}\\
&\leq 2\|u\|_{H^{\mathrm{div}}(\Omega)}\,\|\phi\|_{H^1(\Omega)}
	,\quad u\in\Dscr(\overline\Omega)^n,~\phi\in H^1(\Omega).
\end{aligned}
$$

Denote an arbitrary continuous right inverse of $\gamma_0$ by $\gamma_0^{-r}$, choose an arbitrary $\mu\in\Wscr$, and set $\phi:=\gamma_0^{-r}\mu$. Since $\mu$ vanishes on $\Gamma_0$, we obtain
$$
  \left|\int_{\Gamma_\bullet} (\nu \cdot u)\,\overline\mu\ud\sigma\right|=\left|\int_{\partial\Omega}(\nu\cdot u)\,\overline \gamma_0^{-r}\mu\ud\sigma\right|\leq 2\|u\|_{H^{\mathrm{div}}(\Omega)}\|\gamma_0^{-r}\|\|\mu\|_\Wscr,
$$
i.e., that the restricted normal trace has operator norm at most $2\|\gamma_0^{-r}\|$ from $H^{\mathrm{div}}(\Omega)$ into $\Wscr'$. This restricted normal trace is defined densely in $H^{\mathrm{div}}(\Omega)$ by Lemma \ref{lem:TestDenseHdiv}, and hence it can be extended uniquely to a bounded operator $\gamma_\perp$ from $H^{\mathrm{div}}(\Omega)$ into $\Wscr'$.
\end{proof}

The operator $\gamma_\perp $ is referred to as the \emph{(restricted) normal trace map}.

\begin{theorem}\label{thm:partintext2}
Let $\Omega$ be a bounded Lipschitz set in $\R^n$. For all $f\in H^{\mathrm{div}}(\Omega)$ and $g\in H^1_{\Gamma_0}(\Omega)$ it holds that
\begin{equation}\label{eq:partintext2}
  \Ipdp{\Div f}{g}_{L^2(\Omega)} + \Ipdp{f}{\Grad g}_{L^2(\Omega)^n} 
    = (\gamma_\perp  f,\gamma_0 g)_{\Wscr',\Wscr}.
\end{equation}
In particular, we have the following Green's formula:
\begin{equation}\label{eq:Green2}
  \Ipdp{\Delta h}{g}_{L^2(\Omega)} + \Ipdp{\Grad h}{\Grad g}_{L^2(\Omega)^n} 
    = (\gamma_\perp  \,\Grad h,\gamma_0 g)_{\Wscr',\Wscr},
\end{equation}
which is valid for all $h\in H^1(\Omega)$ such that $\Delta h\in L^2(\Omega)$ and all $g\in H^1_{\Gamma_0}(\Omega)$.
\end{theorem}
\begin{proof}
Since $\pi_0\gamma_0g=0$ for $g\in H^1_{\Gamma_0}(\Omega)$, we obtain from \eqref{eq:partint} that
$$
  \Ipdp{\Div f}{g}_{L^2(\Omega)} + \Ipdp{f}{\Grad g}_{L^2(\Omega)^n} = \Ipdp{\gamma_\perp  f}{\gamma_0 g}_{L^2(\Gamma_\bullet)}
$$
for $f\in H^1(\Omega)^n$ and $g\in H_{\Gamma_0}^1(\Omega)$. Using the fact that $\Wscr'$ is the dual of $\Wscr$ with pivot space $L^2(\Gamma_\bullet)$, we obtain \eqref{eq:partintext2} for $f\in H^1(\Omega)^n$ and $g\in H_{\Gamma_0}^1(\Omega)$.

For every $g\in H_{\Gamma_0}^1(\Omega)$, the mapping $u\mapsto(u,\gamma_0 g)_{\Wscr',\Wscr}$ is a bounded linear functional on $\Wscr'$, and by Proposition \ref{prop:dirichletrange2a}, $\gamma_\perp $ maps $H^{\mathrm{div}}(\Omega)$ continuously into $\Wscr'$. Hence, if $f_n\in H^1(\Omega)^n$ tends to $f$ in $H^{\mathrm{div}}(\Omega)$, then $\Div f_n\to \Div f$ in $L^2(\Omega)$, $f_n\to f$ in $L^2(\Omega)^n$, and $\gamma_\perp  f_n\to \gamma_\perp f$ in $\Wscr'$. We can thus conclude that \eqref{eq:partintext2} holds for all $g\in H^1_{\Gamma_0}(\Omega)$ and all $f$ in the closure of $H^1(\Omega)^n$ in $H^{\mathrm{div}}(\Omega)$, i.e., for all $f\in H^{\mathrm{div}}(\Omega)$; see Lemma \ref{lem:TestDenseHdiv}.

In order to prove \eqref{eq:Green2}, we let $h\in H^1(\Omega)$ be such that $\Delta h\in L^2(\Omega)$ and set $f:=\Grad h$. Then $f\in L^2(\Omega)^n$ and $\Div(\Grad h)=\Delta h\in L^2(\Omega)$, so $f\in H^{\mathrm{div}}(\Omega)$. Now \eqref{eq:Green2} follows from \eqref{eq:partintext2}.
\end{proof}

If we take $\Gamma_0=\emptyset$ in the preceding theorem, then we obtain a well-known special case. The next result gives the surjectivity of the normal trace map, and this critical for associating a boundary triplet to the wave equation.

\begin{theorem}\label{thm:dirichletrange2b}
For a bounded Lipschitz set $\Omega$, $\gamma_\perp $ maps $H^{\mathrm{div}}(\Omega)$ boundedly \emph{onto} $\Wscr'$. 
\end{theorem}
\begin{proof}
By Proposition \ref{prop:dirichletrange2a}, $\gamma_\perp $ maps $H^{\mathrm{div}}(\Omega)$ boundedly \emph{into} $\Wscr'$, and it only remains to establish surjectivity. For this we use an adaptation of the proof of \cite[Cor.\ I.2.]{GiRaBook}. First we fix an arbitrary $\mu\in \Wscr'$ and using the Lax-Milgram theorem \cite[Lemma 2.2.1.1]{GrisBook}, we find a unique $\phi\in H^1_{\Gamma_0}(\Omega)$ which solves the following problem:
\begin{equation}\label{eq:perpsurjprob}
  -\Delta\phi+\phi=0\quad\text{in $L^2(\Omega)
  $}\qquad\text{and}\qquad \gamma_\perp \Grad\phi=\mu.
\end{equation}

Indeed, the sesqui-linear form $(v,\phi)\mapsto\Ipdp{v}{\phi}_{H^1_{\Gamma_0}(\Omega)}$ is bounded and coercive on $H^1_{\Gamma_0}(\Omega)^2$, and the linear form $v\mapsto(\gamma_0v,\mu)_{\Wscr,\Wscr'}$ is bounded on $H^1_{\Gamma_0}(\Omega)$ according to Lemma \ref{lem:dirichletrange}. By the Lax-Milgram theorem there exists a unique $\phi\in H^1_{\Gamma_0}(\Omega)$, such that
\begin{equation}\label{eq:perptrsurjlm}
  \Ipdp{v}{\phi}_{H^1_{\Gamma_0}(\Omega)}=(\gamma_0v,\mu)_{\Wscr,\Wscr'},\quad 
    v\in H^1_{\Gamma_0}(\Omega).
\end{equation}
Taking $v\in\Dscr(\Omega)$, we by Lemma \ref{lem:dirichletrange} and Green's identity \eqref{eq:Green2} obtain that for all $v\in\Dscr(\Omega):$
$$
\begin{aligned}
  0&= \Ipdp{v}{\phi}_{H^1_{\Gamma_0}(\Omega)}=\Ipdp{v}{\phi}_{L^2(\Omega)}+\Ipdp{\Grad v}{\Grad\phi}_{L^2(\Omega)^n} \\
  &= (v,\overline\phi)_{\Dscr(\Omega),\Dscr(\Omega)'}+(\Grad v,\overline{\Grad\phi})_{\Dscr(\Omega)^n,(\Dscr(\Omega)')^n} 
    = (v,\overline{(I-\Delta)\phi})_{\Dscr(\Omega),\Dscr(\Omega)'},
\end{aligned}
$$
i.e., that $\Delta\phi=\phi$ in the sense of distributions on $\Omega$, and hence in particular $\phi\in H^1_{\Gamma_0}(\Omega)$ with $\Delta\phi\in L^2(\Omega)$. Using this and \eqref{eq:Green2} on \eqref{eq:perptrsurjlm}, we thus obtain
\begin{equation}\label{eq:perptrsurjlm2}
\begin{aligned}
 (\gamma_0v,\mu)_{\Wscr,\Wscr'} &=\Ipdp{v}{\phi}_{H^1_{\Gamma_0}(\Omega)} = \Ipdp{v}{\phi}_{L^2(\Omega)}+\Ipdp{\Grad v}{\Grad\phi}_{L^2(\Omega)^n} \\
  &= \Ipdp{v}{(I-\Delta)\phi}_{L^2(\Omega)}+(\gamma_0v,\gamma_\perp  \Grad\phi)_{\Wscr,\Wscr'}\\
  &= (\gamma_0v,\gamma_\perp  \Grad\phi)_{\Wscr,\Wscr'},\quad 
    v\in H^1_{\Gamma_0}(\Omega).
\end{aligned}
\end{equation}
This proves that $\phi$ solves the problem \eqref{eq:perpsurjprob}. Now we set $v:=\Grad\phi$, which lies in $H^{\mathrm{div}}(\Omega)$, because $\Div(\Grad\phi)=\phi$ by \eqref{eq:perpsurjprob}. Furthermore, $\gamma_\perp  v=\mu$ and hence $\gamma_\perp $ maps $H^{\mathrm{div}}(\Omega)$ \emph{onto} $\Wscr'$.
\end{proof}

We can now recover \cite[Thm I.2.5 and Cor.\ I.2.8]{GiRaBook} by taking $\Gamma_0=\emptyset$:

\begin{corollary}\label{cor:normaltracesurj}
The normal trace mapping $u\mapsto \nu \cdot \gamma_0u:\Dscr(\overline\Omega)^n\to L^2(\partial\Omega)$ has a unique continuous extension $\gamma_\perp$ that maps $H^{\mathrm{div}}(\Omega)$ boundedly \emph{onto} $H^{-1/2}(\partial\Omega)$.
\end{corollary}

\section{Two general operator-technical lemmas}

We apply the following lemmas in the proof of Theorem \ref{thm:divgradprops}:

\begin{lemma}\label{lem:closedchar}
Let $T$ be a closed linear operator from $\dom T\subset X$ into $Y$, where $X$ and $Y$ are Hilbert spaces. Equip $\dom T$ with the graph norm of $T$, in order to make it a Hilbert space. Let $R$ be a restriction of the operator $T$. 

The closure of the operator $R$ is $\overline R=T\big|_{\overline{\dom R}}$, where $\overline{\dom R}$ is the closure of $\dom R$ in the graph norm of $T$. In particular, $R$ is a closed operator if and only if $\dom R$ is closed in the graph norm of $T$.
\end{lemma}
\begin{proof}
The following chain of equivalences, where $G(R)=\sbm{I\\R}\dom R$ denotes the graph of $R$, proves that $\overline R=T\big|_{\overline{\dom R}}$:
$$
\begin{aligned}
  \bbm{x\\y}\in G(\overline R) \quad&\overeq{(i)}\quad \exists x_k\in \dom R:~x_k\overto X x,~Rx_k\overto Y y\\
    &\overeq{(ii)} \quad \exists x_k\in \dom R:~x_k\overto X x,~Tx_k\overto Y y\\
    &\overeq{(iii)} \quad \exists x_k\in \dom R:~x_k\overto{\dom T} x,~ Tx=y\\
    &\Longleftrightarrow \quad x\in\overline{\dom R},~Tx=y,\\
\end{aligned}
$$
where we have used that (i): $G(\overline R)=\overline{G(R)}$ by the definition of operator closure, (ii): $G(R)\subset G(T)$, and (iii): $T$ is continuous from $\dom T$ into $Y$ and $\dom T$ is complete.

Now it follows easily that $R$ is closed if and only if $\dom R$ is closed in $\dom T$:
$$
  R=\overline R \quad\Longrightarrow\quad T\big|_{\dom R}= T\big|_{\dom{\overline R}} \quad\Longrightarrow\quad \dom R=\overline{\dom R},
$$
and moreover, assuming instead that $\dom R=\overline{\dom R}$, we obtain that
$$
  R = T\big|_{\dom R} = T\big|_{\overline{\dom R}} = \overline R.
$$
\end{proof}

\begin{lemma}\label{lem:invimages}
Let $\gamma$ be a linear operator from the Hilbert space $\Tscr$ into the Hilbert space $\Zscr$.
\begin{enumerate}
\item  Let $\Rscr$ and $\Rscr'$ be two linear subspaces of $\Tscr$ such that $\Ker\gamma\subset\Rscr\cap\Rscr'$ then 
\begin{equation}\label{eq:domsameass}
  \gamma \, \Rscr=\gamma \, \Rscr' \quad\text{if and only if}\quad \Rscr=\Rscr'.
\end{equation}
\item Let $\Rscr$ be a linear subspace of $\Tscr$ and assume that $\gamma:\Tscr\to \Zscr$ is continuous and surjective with $\Ker\gamma\subset\Rscr$. Then $\gamma\overline{\Rscr}=\overline{\gamma\Rscr}$. Furthermore, $\Rscr$ is closed in $\Tscr$ if and only if $\gamma \Rscr$ is closed in $\Zscr$.
\end{enumerate}
\end{lemma}
\begin{proof}
{\em 1}.\/ First assume that $\gamma \, \Rscr=\gamma \, \Rscr'$ and choose $x\in \Rscr'$ arbitrarily. Then we can find a $\xi\in\Rscr$ such that $\gamma x=\gamma\xi$, and then $x-\xi\in\Ker\gamma\subset \Rscr$ by the assumption $\Ker\gamma\subset\Rscr\cap\Rscr'$, so that $x=x-\xi+\xi\in\Rscr$. This proves that $\Rscr'\subset\Rscr$, and since there is no distinction between $\Rscr$ and $\Rscr'$ in the result statement, it also holds that $\Rscr\subset\Rscr'$. The implication from right to left in \eqref{eq:domsameass} is trivial.

\smallskip\noindent
{\em 2}.\/
Fix $f\in\overline{\Rscr}$ arbitrarily and let $f_k\in \Rscr$ tend to $f$ in $\Tscr$. Then $\gamma f_k\to \gamma f$ in $\Zscr$ by the continuity of $\gamma$, and hence $\gamma f\in \overline{\gamma\,\Rscr}$, i.e., $\gamma \,\overline {\Rscr}\subset \overline{\gamma\,\Rscr}$.

For the converse inclusion, we first remark that since
  $\gamma$ is onto ${\mathcal Z}$, there exists a continuous right
  inverse of $\gamma$. We denote this right-inverse by
  $\gamma^{-r}$. Now let $g\in \overline{\gamma\,\Rscr}$ be arbitrary
  and let $f_k\in \Rscr$ be a sequence such that $\gamma f_k\to g$ in
  $\Zscr$. Then $\gamma^{-r}\gamma f_k\to \gamma^{-r}g=:f'$ in
  $\Tscr$ and thus $\gamma f'=g$. By the definition of a right inverse
  there holds $\gamma(f_k-\gamma^{-r}\gamma
  f_k) =0$, and since $\ker(\gamma) \subset {\mathcal R}$ and $f_k \in
  {\mathcal R}$, we conclude that $\gamma^{-r}\gamma f_k\in\Rscr$ for
  all $k$. So $f'\in\overline\Rscr$ and thus $g =\gamma f' \in
  \gamma\,\overline{\Rscr}$. Combining this with the other inclusion,
  we have established that
  $\gamma\,\overline{\Rscr}=\overline{\gamma\,\Rscr}$.
 We concentrate next on the last assertion. 

 If ${\mathcal R}$ is closed, then by the previous result $\gamma
 {\mathcal R} = \gamma \overline{{\mathcal R}} = \overline{\gamma
   {\mathcal R}}$. Thus $\gamma {\mathcal R}$ is closed.

  If $\gamma {\mathcal R}$ is closed, then $\gamma {\mathcal R} =
  \overline{\gamma {\mathcal R}} = \gamma \overline{{\mathcal R}}$,
  where we used the fist result again. Defining ${\mathcal R}'$ as
  $\overline{{\mathcal R}}$, we see that the left hand-side of
  (\ref{eq:domsameass}) holds. Since $\ker(\gamma) \subset {\mathcal
    R} = {\mathcal R} \cap {\mathcal R}'$, we conclude from part (1)
  that $\Rscr=\overline{\Rscr}$.
\end{proof}

\end{document}